\documentclass[12pt]{amsart}
\usepackage{ amsmath, amsthm, amsfonts, amssymb, color}
 \usepackage{mathrsfs}
\usepackage{amsfonts, amsmath}
 \usepackage{amsmath,amstext,amsthm,amssymb,amsxtra}
 \usepackage{txfonts}
 \usepackage[colorlinks, citecolor=blue,pagebackref,hypertexnames=false]{hyperref}
 \allowdisplaybreaks
 \usepackage{pgf,tikz}
 \usepackage{fancybox}
\usepackage{tikz}
\usepackage{graphicx}
\usepackage{graphics}

\begin{document}
\baselineskip 17pt
\hfuzz=6pt

\newtheorem{theorem}{Theorem}[section]
\newtheorem{proposition}[theorem]{Proposition}
\newtheorem{coro}[theorem]{Corollary}
\newtheorem{lemma}[theorem]{Lemma}
\newtheorem{definition}[theorem]{Definition}
\newtheorem{example}[theorem]{Example}
\newtheorem{remark}[theorem]{Remark}
\newcommand{\ra}{\rightarrow}
\renewcommand{\theequation}
{\thesection.\arabic{equation}}
\newcommand{\ccc}{{\mathcal C}}
\newcommand{\one}{1\hspace{-4.5pt}1}

\newtheorem*{TheoremA}{Theorem A}

\newtheorem*{TheoremB}{Theorem B}

 \def \Lips  {{   \Lambda}_{L}^{ \alpha,  s }(X)}
\def\BL {{\rm BMO}_{L}}
\def\Ma { {\mathcal M} }
\def\MM { {\mathcal M}_0^{p, 2, M, \epsilon}(L) }
\def\dMM { \big({\mathcal M}_0^{p,2, M,\epsilon}(L)\big)^{\ast} }

\def\HHAL { \mathbb{H}^1_{L,{at}, M}(X) }
\def\HAL { H^1_{L,{at}, M}(X) }
\def\HA{ H^1_{L,{at}, 1}(X) }
\def\HSL { H^1_{L }(X) }
\def\HML { H^1_{L, {mol}, M, \epsilon}(X) }
\def\HHML { \mathbb{H}^1_{L, {mol}, M, \epsilon}(X) }
\def\HM{ H^1_{L, {mol}, 1}(X) }
\def\HMH { H^1_{L, {\rm max}, h}(X) }
\def\HNH { H^1_{L, {\mathcal N}_h}(X) }

\def\HSP { H^1_{L, S_P}(X) }
\def\HMP { H^1_{L, {\rm max}, P}(X) }
\def\HNP { H^1_{L, {\mathcal N}_P}(X) }

  \def\RR {  {\mathbb R}^n}
\def\HSL { H^p_{L, S_h}(X) }
\newcommand\mcS{\mathcal{S}}
\newcommand\mcB{\mathcal{B}}
\newcommand\D{\mathcal{D}}
\newcommand\C{\mathbb{C}}
\newcommand\N{\mathcal{N}}
\newcommand\R{\mathbb{R}}
\newcommand\Rf{\mathfrak{R}}
\newcommand\G{\mathcal{G}}
\newcommand\T{\mathbb{T}}
\newcommand\Z{\mathbb{Z}}
\newcommand\lp{L^p(M,\mu)}
\newcommand{\X}{\mathbb{X}}
\newcommand{\PP}{\mathbb{P}}
\newcommand{\cX}{\mathcal{X}}
\newcommand{\cY}{\mathcal{Y}}
\newcommand{\intav}{-\hspace{-0.05cm}\!\!\!\!\!\!\int}
\newcommand{\E}{\mathbb{E}}
\newcommand{\inti}{\frac{1}{|I|} \, \int_I \,}
\newcommand{\al}{\alpha}
\newcommand{\om}{\omega}
\newcommand{\Om}{\Omega}
\newcommand{\s}{\smallskip}
\newcommand{\m}{\medskip}
\newcommand{\ch}{\text{ch}}
\newcommand{\bb}{\bigskip}
\newcommand{\e}{\varepsilon}
\newcommand{\II}{\text{II}}
\newcommand{\I}{\text{I}}
\newcommand{\vmo}{{\rm VMO}}
\newcommand{\BMO}{{\rm BMO}}
\newcommand{\VMO}{{\rm VMO}}
\newcommand{\La}{\Lambda}

\newcommand\CC{\mathbb{C}}
\newcommand\NN{\mathbb{N}}
\newcommand\ZZ{\mathbb{Z}}

\newcommand\bmo{{\rm BMO}_L(X)}

\renewcommand\Re{\operatorname{Re}}
\renewcommand\Im{\operatorname{Im}}

\newcommand{\mc}{\mathcal}

\def\SL{\sqrt[m] L}
\newcommand{\la}{\lambda}
\def \l {\lambda}
\newcommand{\eps}{\varepsilon}
\newcommand{\pl}{\partial}
\newcommand{\supp}{{\rm supp}{\hspace{.05cm}}}
\newcommand{\x}{\times}
\newcommand{\mar}[1]{{\marginpar{\sffamily{\scriptsize
        #1}}}}

\newcommand\wrt{\,{\rm d}}

\title[The Garnett--Jones Theorem on $\BL$  ]
{The Garnett--Jones Theorem    on   BMO spaces  associated with operators
  and   applications
  }

\author[P. Chen,  X.T. Duong,  J. Li, L. Song and L.  Yan]{Peng Chen, \
	 Xuan Thinh Duong,
	Ji Li, \ Liang Song \ and \ Lixin Yan
}
\address{Peng Chen, Department of Mathematics, Sun Yat-sen (Zhongshan)
	University, Guangzhou, 510275, P.R. China}
\email{chenpeng3@mail.sysu.edu.cn}
\address{Xuan Thinh Duong, Department of Mathematics, Macquarie University, NSW 2109, Australia}
\email{xuan.duong@mq.edu.au}
\address{Ji Li, Department of Mathematics, Macquarie University, NSW, 2109, Australia}
\email{ji.li@mq.edu.au}
\address{Liang Song, Department of Mathematics, Sun Yat-sen (Zhongshan)
	University, Guangzhou, 510275, P.R. China}
\email{songl@mail.sysu.edu.cn}
\address{
	Lixin Yan, Department of Mathematics, Sun Yat-sen (Zhongshan) University, Guangzhou, 510275, P.R. China}
\email{mcsylx@mail.sysu.edu.cn
}

\date{\today}

\begin{abstract}
Let $X$ be a metric space with doubling measure, and
 $L$ be a nonnegative self-adjoint operator on $L^2(X)$ whose heat kernel satisfies the Gaussian upper bound.
 Let $f$ be in the space  $ {\rm BMO}_L(X)$   associated with the operator $L$ and we define
 its distance from the subspace $L^{\infty}(X)$ under the  $ {\rm BMO}_L(X)$ norm as follows:
$$
	{\rm dist} (f, L^{\infty}):=	 \inf_{g\in L^{\infty}}  \|f -g\|_{{\rm BMO}_L(X)}.
$$
In this paper we prove that ${\rm dist} (f, L^{\infty})$ is equivalent to the infimum of  the constant $\varepsilon$  in  the  John--Nirenberg inequality for the space ${\rm BMO}_L(X)$:
$$
	\sup_B { \mu\big(\{ x\in B: |f(x)-e^{-{r_B^2}L}f(x)|>\lambda\}\big)   \over \mu(B)}  \leq e^{-\lambda/\varepsilon}\ \ \ \ {\rm for\ large\ } \lambda.
$$
This  extends the well-known result of  Garnett and Jones \cite{GJ1}  for the classical $\BMO$ space (introduced by John and Nirenberg).
As an application, we  show   that a $\BMO_L(X)$ function with compact support can be decomposed as the summation of an $L^\infty$-function and the integral of the heat kernel (associated with $L$) against a finite Carleson measure on $X\times[0,\infty)$. The key new technique is a geometric construction involving the semigroup $e^{-tL}$. We also resort to several fundamental tools including the stopping time argument and the random dyadic lattice.
\end{abstract}

\maketitle


\section{Introduction}
\setcounter{equation}{0}

\noindent
{\bf 1.1. Background.} \  Let $B$ denote a ball in the Euclidean space ${\mathbb R^n}$ and
$|B|$ denote the Lebesgue measure of $B$. We say that a locally integrable real-valued function
$f$ on ${\mathbb R^n}$
has bounded mean oscillation, i.e. $f\in {\rm BMO}(\RR)$, if
\begin{eqnarray}\label{e1.1}
\|f\|_{\ast}:=\sup_B {1\over |B|}\int_B |f(x)-f_B|\, dx <\infty,
\end{eqnarray}
where
\begin{eqnarray}\label{e1.2}
 f_B :={1\over |B|}\int_B  f(x) \, dx.
\end{eqnarray}
It is well-known that $L^{\infty}(\mathbb R^n)\subsetneq  {\rm BMO}(\RR)$.
In \cite{GJ1}  Garnett and Jones gave comparable upper and lower bounds for the distance
\begin{eqnarray}\label{e1.3}
{\rm dist} (f, L^{\infty}):=	 \inf_{g\in L^{\infty}}  \|f -g\|_{\ast}
\end{eqnarray}
by the infimum of  the constant $\varepsilon$ in  the   John--Nirenberg inequality \cite{JN}. To be more precise,
consider John--Nirenberg inequality
\begin{eqnarray}\label{e1.4}
	\sup_B { \left|\{ x\in B: |f(x)-f_B|>\lambda\}\right|  \over |B|}  \leq e^{-\lambda/\varepsilon}\ \ \ \ \ \  {\rm for\  } \lambda \ {\rm  \ large},
\end{eqnarray}
and define
$$
\varepsilon(f)=\inf \{\varepsilon>0:   \eqref{e1.4} \ {\rm holds}\}.
$$
 Garnett--Jones \cite{GJ1}  states that  there are positive  constants $c_1$ and $c_2$  depending only on the dimension  $n$ such that
 \begin{eqnarray}\label{e1.5}
	c_1 \varepsilon(f)\leq \inf_{g\in L^{\infty}}  \|f -g\|_{\ast}  \leq c_2 \varepsilon(f).
\end{eqnarray}
  As pointed out in \cite{GJ1}, a generalization of \eqref{e1.5} on space of homogeneous type \cite{CW} can be  proved with minor variations in the argument (we refer to \cite{J0}).

The   Garnett--Jones theorem  has several important applications.
One example is on the decomposition of BMO functions. Note that  Fefferman--Stein \cite{FS} showed
that
   $f\in {\rm BMO}(\RR)$ if and only if
\begin{eqnarray}\label{e1.6}
	f=g_0 +\sum_{j=1}^n R_j g_j,
\end{eqnarray}
where $R_j$ is the $j$-th Riesz transform, and $g_0, g_1, \ldots, g_n\in L^{\infty}$ with $\|f\|_{\ast}\sim \inf \Big\{\sum\limits_{j=0}^n \|g_j\|_{\infty}: \eqref{e1.6}  \ {\rm holds}\Big\}$.
The  result \eqref{e1.5} can be reformulated as
\begin{eqnarray}\label{e1.7}
	\varepsilon(f) \sim \inf \Big\{ \sum_{j=1}^n \|g_j\|_{\infty}:
	\eqref{e1.6}  \ {\rm holds\ for\ some\ }  g_0\in L^{\infty}  \Big\}.
\end{eqnarray}
Another application is a higher dimensional  Helson--Szeg\"o theorem, see \cite[Corollary 1.2]{GJ1}. In \cite{J3},
 Jones also used it to     derive  an estimate for corona solutions. Bourgain \cite{Bou} applied it in the study of embedding $L^1$ in $L^1/H^1$, proving that $L^1$ is isomorphic to a subspace of $L^1/H^1$.
In the last few decades, the    Garnett--Jones theorem has been  studied extensively, see for example
  \cite{Bou, G, GJ1, GJ2, J1, J2, J3,  U1, V, W} and  the references therein.

\medskip

\noindent
{\bf 1.2. Main results}. \ Our setting throughout the paper is as follows.

\noindent Let $(X, d, \mu)$ be a metric
measure space with  $\mu$ satisfying the doubling  condition
\begin{eqnarray}\label{doubling2}
	\mu(B(x, \lambda r)) \le C \lambda^n \mu(B(x,r)), \ \,\;  \forall x \in X,\, \lambda>1, \, r > 0,
\end{eqnarray}
where $C$ and $n$ are positive constants and $\mu(B(x,r))$ denotes the volume of the open ball
$B(x,r)$ of center $x$ and radius $r$.

\noindent Let  $L$ be an operator on $L^2(X)$ satisfying
  the following properties:
\begin{enumerate}
	\item [\textbf{(H1)}]  $L$ is a densely defined, nonnegative self-adjoint operator on $L^2(X)$;
	\item [\textbf{(H2)}] The kernel of $e^{-tL}$, denoted by ${h}_{t}(x,y)$, is a measurable function on
	$X\times X$ and satisfies a Gaussian upper bound, that is,
	$$
	\left|{ h}_{t}(x,y)\right| \leq  {C\over   \mu(B(x,\sqrt{t}))} \exp\left(-{d(x,y)^2\over ct}\right)
	$$
	for all $t>0$ and $x,y\in X$, where $C$ and $c$ are positive constants.
\end{enumerate}

Next, we recall the new BMO space introduced in \cite{DY1}.
For some $x_0\in X$ and $\beta>0$, let $\mathcal M_{x_0,\beta}$ be the set of functions defined as
$$ \mathcal M_{x_0,\beta}:=\left\{ f\in L^1_{\rm loc}(X): \int_X \frac{|f(x)|}{(1+d(x_0,x))^\beta\mu(B(x_0,1+d(x_0,x)))}d\mu(x)<\infty\right\} .
$$
Then denote by $\mathcal M$ the set of functions
$$
\mathcal M=\bigcup_{x_0\in X}\bigcup_{\beta>0}\mathcal M_{x_0,\beta}.
$$
In \cite{DY1}, Duong and Yan
  introduced and
developed a new function  space ${\rm BMO}_L(X)$   associated with an operator $L$
by using a maximal function introduced by Martell
in \cite{Ma}.
Given an operator $L$
  satisfying \textbf{(H1)} and  \textbf{(H2)},
  the key idea in \cite{DY1} is to view $
e^{-tL}f$ as an average version of $f$  (at the scale $t$) and use the quantity
\begin{eqnarray}	\label{e1.8}
e^{-{r^2_B}L}	f(x)=\int_{X} h_{{r^2_B}}(x,y)f(y) \, d\mu(y)
\end{eqnarray}
to replace the average $f_B$ in the definition (\ref{e1.2}) of the
classical  $\BMO$ space.
We then  say that $f\in \mathcal M$ is in ${\rm BMO}_L(X)$, the {\rm BMO} space associated with $L$,
  if
\begin{eqnarray}	\label{e1.9}
\|f\|_{\BL(X)}:=	 \sup_B{1\over \mu(B)}\int_{B}|f(x)-e^{-{r^2_B}L}	 f(x)| \, d\mu(x)<\infty.
\end{eqnarray}
It was proved in \cite[Section 2]{DY1} that the classical ${\rm BMO}(\mathbb R^n)$ coincides with the one associated with the standard Laplacian  $\Delta$ on $\mathbb R^n$, that is, ${\rm BMO}(\mathbb R^n)= {\rm BMO}_\Delta(\mathbb R^n)$.
When the kernel $h_t(x,y)$ of $e^{-tL}$ satisfies
$
\int_{  X} h_t(x,y)d\mu(y)=1 $ for  almost all $x\in X$,
one may readily verify that ${\rm BMO}(X)\subseteq {\rm BMO}_L(X)$.
Depending on the choice of the
operator $L$  such as  Schr\"odinger operators   or the divergent form operators,  ${\rm BMO}_L(X)$ can be substantially  different from ${\rm BMO}(X)$, see for examples,    \cite{ACDH, DDSY, DY2, DGMTZ, HM}
and the references therein.
 Many important features of the classical $\BMO$ space are retained
by the new ${\rm BMO}_L$ spaces such as the John--Nirenberg inequality,   complex interpolation and
duality  (see \cite{DGMTZ, DY1, DY2, HM}).  To be more specific, Duong and Yan \cite[Theorem 3.1]{DY1} established   the John--Nirenberg inequality as follows:  $f\in {\rm BMO}_L(X)$ if and only if there exists $\varepsilon>0$ and $\lambda_0=\lambda_0(\varepsilon, f)$ such that
\begin{eqnarray}\label{e1.11}
	\sup_B { \mu(\{ x\in B: |f(x)-e^{-{r_B^2}L}f(x)|>\lambda\} )  \over \mu(B)}  \leq e^{-\lambda/\varepsilon}
\end{eqnarray}
whenever $\lambda>\lambda_0$.

Parallel to the classical $\BMO$,  it is direct to see that $L^{\infty}(X)\subsetneq{\rm BMO}_L(X)$. Thus,
we raise the following question: given  $f\in {\rm BMO}_L(X)$, what is its distance from the subspace $L^{\infty}(X)$ under the ${\rm BMO}_L(X)$ norm?

 In this article we will give comparable upper and lower bounds  for the distance
\begin{eqnarray}\label{e1.10}
	{\rm dist} (f, L^{\infty}):=	 \inf_{g\in L^{\infty}}  \|f -g\|_{{\rm BMO}_L(X)},
\end{eqnarray}
which are expressed in terms of one constant
in the John--Nirenberg inequality for the ${\rm BMO}_L(X)$ space.
 Define
$$
\varepsilon_L(f)=\inf \Big\{\varepsilon>0:   \eqref{e1.11} \ \  {\rm holds\ for\  all\ } \lambda\geq \lambda_0(\varepsilon, f)\Big\}.
$$
By the John--Nirenberg inequality \cite{DY1},
$f\in {\rm BMO}_L(X)$ if and only if there exists $\varepsilon >0$ such that \eqref{e1.11} holds, and   in fact, $\varepsilon=c\|f\|_{\BL(X)}$,  where $\lambda_0=c'\|f\|_{\BL(X)}$.
If $f\in L^{\infty}$, then $f$ satisfies \eqref{e1.11} for all $\varepsilon>0$, and thus $\varepsilon_L(f)=0$.

The first aim of this article is to establish  the following result.
  \begin{theorem}\label{th1.1}
  	Suppose that $L$ is an operator
  	on $L^2(X)$ satisfying \textbf{(H1)} and \textbf{(H2)}.
  	Let $  f\in  {\rm BMO}_L(X)$.  Then
 	there are constants $c_1$ and $c_2$  depending only on $C$ and $n$ in \eqref{doubling2}  such that
 	\begin{eqnarray}\label{e1.12}
 		c_1 \varepsilon_L(f)\leq \inf_{g\in L^{\infty}}  \|f -g\|_{\BL(X)}  \leq c_2 \varepsilon_L(f).
 	\end{eqnarray}
 \end{theorem}

The main idea to prove this result is as follows:

(1) the first inequality  in \eqref{e1.12} follows from applying \eqref{e1.11} to $f-g$;

(2) the second inequality  in \eqref{e1.12} is proved by establishing the following result:  for
 every
$\varepsilon>\varepsilon(f)$,   there exist two positive constants $C_1$ and $C_2$ such that
 	\begin{align}\label{e1.13}
 	f(x)=g(x)+h(x)
 	\end{align}
 	with
 	$$
 	\|g\|_{L^\infty}\leq C_1\lambda_0 \ \ \ \ {\rm and}\ \  \ \
 	\|h\|_{\BL(X)}\leq C_2 \varepsilon.
 	$$
 However, if one attempts to establish \eqref{e1.13} by using the proof of Garnett--Jones \cite{GJ1, GJ2} for the classical  ${\rm BMO}$ case,   one immediately runs into essential difficulties. 	
 			Unlike the classical BMO space \cite{GJ2, JN},  we do not have a theory of  dyadic $\BMO_L$  		associated with operator $L$.	
Another  fact is  that under a dyadic system,   the standard  Calder\'on--Zygmund stopping time argument  gives  that for a large dyadic cube $Q_0$,
$$
\left(f(x)-e^{-\ell(Q_0)^2L}f(x)\right)\chi_{Q_0}(x)=g_{Q_0}(x)+h_{Q_0}(x),
$$
where $g_{Q_0}\in L^\infty$ and
\begin{align}\label{e1.133}
	h_{Q_0}(x)=\sum_Q a_Q(x)\chi_Q(x),
\end{align}
see Proposition~\ref{lem3.2} below.  The technical difficulty we encounter,  is that for the function $h_{Q_0}(x)$ in \eqref{e1.133}, the coefficient $a_Q(x)$ is a function,    the support of $a_Q(x)$ is  not compact and $h_{Q_0}(x)$   oscillates too much. So, the function $h_{Q_0}(x)$  does not belong to the space ${\rm BMO}_L(X)$. The    main idea of our proof    is to use the method of  random dyadic lattice method  and geometric  construction
from \cite{GJ2} (see also \cite{H, NTV}), and the fact that the the family of dyadic cubes   $\{ Q\}$ is  sparse  allows us to  add   more $L^\infty$ functions in   $h_{Q_0}(x)$  to lower the oscillation.    Our proof of~\eqref{e1.13}  will  be   done constructively in  Sections   3, 4  and 5 below.

The second    aim of this article is  to use the construction
 in proving \eqref{e1.12} to
give  a Carleson's characterization of the ${\rm BMO}_L$
space, which extends the result for the classical John--Nirenberg $\BMO$ space due to   Carleson \cite{C} in 1976.
Recall that a
  measure $\sigma$ defined on $X\times {\mathbb R^+}
 $ is said to be a
Carleson measure if there is a positive constant $c_{\sigma}$ such that
for each ball $B$
on $X$,
$$
\sigma({\widehat B})\leq c_{\sigma}\mu(B),
$$
where ${\widehat B}=\{(x,t)\in X\times {\mathbb R^+}: x\in B, \
0<t<r_B\}$, and $r_B$
is the radius of the ball $B$.
Our   result can be stated as in the following.

\begin{theorem}\label{th1.3}   Suppose that $L$ is an operator
on $L^2(X)$ satisfying \textbf{(H1)} and \textbf{(H2)}.
	Suppose  that $f\in {\rm BMO}_L(X)$ has compact support. Then there exist $g\in L^\infty(X)$ and a finite Carleson measure $\sigma$ such that
	$$
	f(x)=g(x)+\int_{X\times [0,\infty)} h_{t^2}(x,y)\,d\sigma(y,t),
	$$
	in which $h_{t^2}(x,y)$ is the function in condition \textbf{(H2)} and $\|g\|_\infty+\|\sigma\|_{\mathcal C}\leq C\|f\|_{\BL(X)}$.
\end{theorem}

Comparing our Theorem \ref{th1.3} to known results of the classical $\BMO(\RR)$ spaces, we note  that Carleson showed in 1976 that for
every $f$ in the classical $\BMO(\RR)$ with compact support,  there exist $g\in L^\infty(\R^n)$ and a finite Carleson measure $\sigma$ such that
$$
f(x)=g(x)+\int_{\R^n\times [0,\infty)} t^{-d}K((x-y)/t)\,d\sigma(y,t)
$$
with $\|g\|_{L^\infty}+\|\sigma\|_{\mathcal C}\leq C\|f\|_{\BMO}$,
where  the kernel $K(x)$ is in $L^1(\R^n)$ and is assumed to satisfy $\int K(x)dx=1$,  $|K(x)|  \leq (1+|x|)^{-n-1}$ and  $|\nabla K(x)|\leq (1+|x|)^{-n-2}$,
see also the work of A. Uchiyama \cite{U}, and   of Garnett--Jones \cite{GJ2}.  The proofs in \cite{C, GJ2, U}
are based on an iteration argument.
In \cite{W} Wilson  gave  a new proof of Theorem~\ref{th1.3} when $K$ is the Poisson kernel, i.e,
\begin{eqnarray}\label{ep}
	p(x)=c_n { (1 +|x|^2)^{-{n+1\over 2}}}, \ \ \ \ c_n={\Gamma\Big({n+1\over 2}\Big) \over \pi^{(n+1)/2}}
\end{eqnarray}
by using the Poisson semigroup property  and Green's theorem
to avoid the iteration and made the construction much more explicit.
Recently,   Theorem~\ref{th1.3}  was proved in \cite{CDLSY} in the special case that when
$L=-\Delta+V$
is a Schr\"odinger operator  on $L^2(\mathbb R^n)$, $n\geq3$, where the nonnegative
potential $V$ belongs to  the reverse H\"older class    $B_q$ for some $q\geq n.$

\medskip

\noindent
{\bf 1.3. Comments on the results and methods of the proof.} \
In our Theorems~\ref{th1.1} and ~\ref{th1.3},  we assume only an upper
bound on  the  kernel $h_t(x,y)$ of $e^{-tL}$. There is no regularity condition on spatial variables $x$ or $y$
(e.g.  H\"older continuity or H\"ormander condition
on the space variables $x$ or $y$). Another feature of our results
is that we do not assume the conservation property of the semigroup
 \begin{eqnarray*}
	e^{-tL}1  =1, \ \ \ t>0,
\end{eqnarray*}
thus the classical  $\BMO$  theory \cite{GJ1, GJ2} are not applicable
 in the proofs of Theorems~\ref{th1.1} and ~\ref{th1.3}. 
The lacking of regularity on the spatial variables and conservation property  is indeed the  main obstacle in this article and makes our theory quite subtle and delicate.

Our method allows Theorems~\ref{th1.1} and ~\ref{th1.3} to be  applicable
to a larger class of  operators $L$, for example $L = \mathcal{L} + V$ with a general non-negative potential $V$ on
a doubling space $X$ whenever $\mathcal{L}$ is an operator which satisfies conditions \textbf{(H1)} and \textbf{(H2)}.
It is now understood that there are important examples where the
classical $\BMO$ theory is not applicable, and these situations, being tied to
the theory of partial differential operators beyond the Laplacian, return us in
some sense to
  the original point of view of \cite{JN}. That is, we shall consider the
space ${\rm BMO}_L(X)$  adapted to a linear operator $L$, in much the same way that the
classical $\BMO$ space  is adapted to the Laplacian. On the other hand, the real variable techniques of \cite{FS, GJ1, GJ2, JN, St2} are still of
fundamental importance to us here.

 The layout of the article is as follows. In Section 2 we recall  random dyadic lattices on spaces of homogeneous type. The proof of Theorem~\ref{th1.1} will be given
 in  Sections   3, 4  and 5.  The proof of Theorem~\ref{th1.3} will be given in Section 6.
 Finally, we conclude this article in Section 7 by
giving some remarks and comments
 for further study related to Theorems~\ref{th1.1} and \ref{th1.3} on the space ${\rm BMO}_L(X)$ associated with operators.

\medskip

\section{Random dyadic lattices on spaces of homogeneous type}
\setcounter{equation}{0}

In this  section we recall the dyadic cubes in spaces of
homogeneous type, and the random dyadic lattices of cubes.
 As shown in~\cite{HK, HT}, building
on~\cite{Chr}, there exists a dyadic system $\D$ for this
space~$X$. There exist positive absolute constants $c_1$, $C_1$
and $0<\delta<1$ such that we can construct a set of points
$\{x_\alpha^k\}_{k,\alpha}$ and families of sets
$\{Q_\alpha^k\}_{k,\alpha}$ in $X$ satisfying the following
properties:
\begin{eqnarray}
    &&\mbox{if } \ell\leq k, \mbox{ then either } Q_\alpha^k\subset
        Q_\beta^\ell \mbox{ or } Q_\alpha^k
        \cap Q_\beta^\ell=\emptyset;\label {DyadicP1}\\
    &&\mbox{for every } k\in \Z \mbox{ and } \alpha\neq\beta,  Q_\alpha^k\cap
       Q_\beta^k=\emptyset; \label {DyadicP2}\\
    &&\mbox{for every } k\in \Z,\,
        X=\cup_\alpha Q_\alpha^k;\label {DyadicP3} \\
    && B(x_\alpha^k,c_1\delta^k)\subset Q_\alpha^k\subset
        B(x_\alpha^k,C_1\delta^k);\label {prop_cube4}\\
    &&\mbox{if } \ell\leq k \mbox{ and } Q_\alpha^k\subset Q_\beta^\ell,
        \mbox{ then } B(x_\alpha^k,C_1\delta^k)\subset
        B(x_\beta^\ell,C_1\delta^\ell).\label {DyadicP5}
\end{eqnarray}
Here for each $k\in\mathbb{Z}$, $\alpha$ runs over an
appropriate index set. We call the set $Q_\alpha^k$ a
\emph{dyadic cube} and $x_\alpha^k$ the \emph{center} of the
cube. Also, $k$ is called the \emph{level} of this cube. We
denote the collection of dyadic cubes at level $k$ by $\D^k$,
and the collection of all dyadic cubes by~$\D$. When
$Q_\alpha^k\subset Q_\beta^{k-1}$, we say $Q_\alpha^k$ is a
child of $Q_\beta^{k-1}$ and $Q_\beta^{k-1}$ is the parent
of~$Q_\alpha^{k}$. Because $X$ is a space of homogeneous type,
there is a uniform constant~$\mathcal{N}$ such that each cube
$Q\in\D$ has at most~$\mathcal{N}$ children.

Denote
$$
E_j(Q_\alpha^k):=\big\{Q_\beta^{k+j}\subset Q_\alpha^k: d(Q_\beta^{k+j}, X\backslash Q_\alpha^k)\leq C\delta^{k+j}\big\}
$$
and
$$
e_j(Q_\alpha^k):=\big\{x:x\in Q_\beta^{k+j}\, \mbox{for\, some}\, Q_\beta^{k+j}\in E_j(Q_\alpha^k)\big\}.
$$
We have the following fact (\cite[p. 12, Proof of (3.6)]{Chr}) that there exists an absolute constant $\eta>0$ such that
\begin{align}\label{annulars1}
\mu(e_j(Q_\alpha^k))\leq C\delta^{\eta j} \mu(Q_\alpha^k).
\end{align}

Denote by $Q_{\alpha,t}^k$  the  annulus of a dyadic cube $Q_\alpha^k$, that is
$Q_{\alpha,t}^k:=\{x\in X\backslash Q_\alpha^k: d(x, Q_\alpha^k)\leq t\delta^k\}$. Then it follows from \eqref{annulars1} and the doubling property that
\begin{align}\label{annulars2}
\mu(Q_{\alpha,t}^k)\leq Ct^\eta \mu(Q_\alpha^k).
\end{align}
Actually note that the set $\{x\in Q_\alpha^k: d(x,Q_\alpha^k)\leq t\}\subset e_j(Q_\alpha^k)$, see see the proof of (3.6) in \cite[p. 12]{Chr}. Then $Q_{\alpha,t}^k$ can be covered by the union of $10 Q_\beta^{k+j}$ where $Q_\beta^{k+j}\subset E_j(Q_\alpha^k)$ for $\delta^j\sim t$. By the doubling property,
$$
\mu(Q_{\alpha,t}^k)\leq C10^n \sum_{\beta} \mu(Q_\beta^{k+j})\leq C \mu(e_j(Q_\alpha^k))\leq Ct^\eta \mu(Q_\alpha^k).
$$

In~\cite[Theorem 5.2]{HT} Hyt\"onen and Tapiola constructed
random dyadic lattices on spaces of homogeneous type, extending
an earlier result of Nazarov, Treil and Volberg~\cite{NTV} and Hyt\"onen and Kairema~\cite{HK}.
Specifically, there exists a probability space $(\Omega,\PP)$
such that for each $\omega\in\Omega$ there is an associated
dyadic lattice $\D(\omega) = \{Q_\alpha^k(\omega)\}_{k,\alpha}$
related to dyadic points $\{x_\alpha^k(\omega)\}_{k,\alpha}$,
with the properties \eqref{DyadicP1}--\eqref{DyadicP5} above,
and the following property holds: there exist absolute
constants $C$, $\eta > 0$ such that
\begin{eqnarray}\label {Rcondition}
    \PP\left(\{\omega\in\Omega: x\in \bigcup_{\alpha} \partial_\varepsilon Q_\alpha^k(\omega) \}\right)
    \leq C\left(\frac{\varepsilon} {\delta^k}\right)^\eta
\end{eqnarray}
for all $x\in X,k\in \mathbb Z$ and $\varepsilon>0$. Here $\partial_\varepsilon A$ denotes the $\varepsilon$-boundary of a set $A\subset X$ defined by
$$
\partial_\varepsilon A:=\{x\in A: d(x,A^c)<\varepsilon\}\cup \{x\in A^c: d(x,A)<\varepsilon\}.
$$
Let $\E_\omega$ denote the expectation taken over $\omega\in \Omega$ of a family of functions $f^{(\omega)}$ indexed by $\omega\in \Omega$
$$
\E_\omega f^\omega(x)=\int_\Omega f^\omega(x)d\omega.
$$

Throughout the paper, for the sake of simplification of the natation, we always assume $\delta=2^{-1}$ and all the proofs can be checked directly for other $\delta$.
\bigskip

\section{Proof of Theorem~\ref{th1.1}: step one}
\setcounter{equation}{0}

In this section we will apply   the  John--Nirenberg inequality \eqref{e1.11}
  to build
  a decomposition   of $f\in {\rm BMO}_L(X)$ under a dyadic system
  by using a stopping time argument which resembles the proof of the John--Nirenberg theorem \cite{DY1}.

To state our   result,  we need to introduce the following {\em graded sequence} and  {\em graded function}.
Let $0<\gamma<1$. A sequence
$$
E_1\supset E_2\supset \ldots \supset E_k\supset \ldots
$$
is said to be $\gamma$-graded, if
\begin{itemize}
	\item[(1)]
 Each $E_k$ is the union of some disjoint dyadic cubes, say $E_k=\cup_j Q_j^{(k)}$;

	\item[(2)] Each $Q_j^{(k)}$ must be contained in some $Q_{j'}^{(k-1)}$;

	\item[(3)] For any $Q_j^{(k-1)}$, $|E_k \cap Q_j^{(k-1)}|\leq \gamma |Q_j^{(k-1)}|$.
\end{itemize}
 With conditions above, the $\gamma$-graded function is the form
$$
h(x)=\sum_k \sum_j a_{Q_j^{(k)}}(x)\chi_{Q_j^{(k)}}(x).
$$

In this section, we prove  the following result.

\begin{proposition}\label{lem3.2}
Fixed a dyadic system $\mathcal D$ and large dyadic cube $Q_0\in \mathcal D$. There is a constant $C$, depending only on the doubling constants in \eqref{doubling2} and the Gaussian upper bound, such that if $f\in \BL(X)$, $\varepsilon>\varepsilon_L(f)$ and $\lambda>\lambda_0(\varepsilon,f)$, then
$$
\big(f(x)-e^{-\ell(Q_0)^2L}f(x)\big)\chi_{Q_0}(x)=g_{Q_0}(x)+h_{Q_0}(x)
$$
with
$$
\|g_{Q_0}\|_{L^\infty}\leq C\lambda
$$
and
$$
 h_{Q_0}(x)=\sum_Q a_{Q}(x)\chi_{Q}(x),
$$
where $Q$ runs over a $\gamma$-graded sequence with $\gamma=2^{-\lambda/\varepsilon}$ and $a_{Q}(x)$ satisfies that
for almost all $x\in 2Q$,
%
\begin{align}\label{size condition3.1}
|a_{Q}(x)|\leq C\lambda,
\end{align}
and
\begin{align}\label{smooth condition3.3}
|La_Q(x)|\leq C\ell({Q})^{-2}\lambda.
\end{align}

\end{proposition}

\begin{proof}
	Let  $f\in \bmo$.  It follows from  \cite[Proposition 2.6]{DY1} that  for every $t>0$ and every $K\geq 1$, there exists a constant
	$c>0$ such that for almost all $x\in X$,
	\begin{eqnarray}\label{bmo1}
		 |(e^{-tL}-e^{-KtL})f(x)|\leq c(1+\log K)\|f\|_{\bmo}.
	\end{eqnarray}
Now we fix the constant $\lambda>C\|f\|_{\bmo}$ such  that
\begin{align*}
	 &\frac{1}{\mu(2Q_0)}\int_{2Q_0}|f(y)-e^{-\ell(Q_0)^2L}f(y)|d\mu(y)\\
	&\leq \frac{1}{\mu(2Q_0)}\int_{2Q_0}\Big(|f(y)-e^{-4\ell(Q_0)^2L}f(y)|+|e^{-4\ell(Q_0)^2L}f-e^{-\ell(Q_0)^2L}f|\Big)d\mu(y)\\
	&\leq C\|f\|_{\bmo}\\
	&<\lambda.
\end{align*}
Then we make the Calder\'on--Zygmund decomposition of $Q_0$ for the function
$(f-e^{-\ell(Q_0)^2L}f)$  relative to $\lambda$.  We first obtain dyadic cubes $Q_1$ which are maximal dyadic subcubes of
$Q_0$ satisfying
\begin{eqnarray}\label{selection0}
 \lambda <	 \frac{1}{\mu(2Q_1)}\int_{2Q_1}|f(y)-e^{-\ell(Q_0)^2L}f(y)|d\mu(y).
\end{eqnarray}
Denote by
$$
\mathcal{G}_1:=\{Q_1\in \D: Q_1\subset Q_0, {\rm maximal\, and\, satisfies\, \eqref{selection0}}\}.
$$
So we have that
\begin{align}\label{de}
	(f-e^{-\ell(Q_0)^2L}f)\chi_{Q_0}
	 &=\left((f-e^{-\ell(Q_0)^2L}f)\chi_{Q_0\backslash\cup_{Q_1\in\G_1} Q_1} +
\sum_{Q_1\in\G_1} (e^{-\ell(Q_1)^2L}-e^{-\ell(Q_0)^2L}f) \chi_{Q_1}\right)\nonumber\\
	&\quad + \sum_{Q_1\in\G_1} (f-e^{-\ell(Q_1)^2L}f) \chi_{Q_1}.
\end{align}
For each $Q_1\in\G_1$   in \eqref{de} above, we make the Calder\'on--Zygmund decomposition for the
function $(f-e^{-\ell(Q_1)^2L}f)$ relative to $\lambda$. Thus we obtain a family
of dyadic cubes $\{Q_2\}$ which are  subcubes of $Q_1$ and satisfy
\begin{eqnarray}\label{selection1}
	\lambda <	 \frac{1}{\mu(2Q_2)}\int_{2Q_2}|f(y)-e^{-\ell(Q_1)^2L}f(y)|d\mu(y).
\end{eqnarray}
Now we put
together all the families $Q_2$  corresponding to different  $Q_1'$s to a new set $\G_2$. That is,
$\mathcal{G}_2:=\{Q_2\in\D: \ {\rm for \  some\, } Q_1,\, Q_2\subset Q_1, {\rm maximal\, and\, satisfies\, \eqref{selection1}}\}$.

 Subsequently,  for each natural number $k\in{\mathbb N}$ we obtain
a family of non-overlapping cubes $Q_k$ such that
\begin{eqnarray}\label{eq3.1}
	 \lambda<\frac{1}{\mu(2Q_k)}\int_{2Q_k}|f(y)-e^{-\ell(Q_{k-1})^2L}f(y)|d\mu(y).
\end{eqnarray}
And then denote by
\begin{align}\label{Gk}
\mathcal{G}_k:=\{Q_k\in \D: {\rm for\, some\, } Q_{k-1},\, Q_k\subset Q_{k-1}, {\rm maximal\, and\, satisfies\, \eqref{eq3.1}}\}.
\end{align}
Denote by $\G_0:=\{Q_0\}$ and ${Q_k^\ast}\in \G_{k-1}$ the dyadic cube containing $Q_k\in \G_k$.
Now
\begin{align*}
	(f-e^{-\ell(Q_0)^2L}f)\chi_{Q_0}&=\sum_{k=1 }^{K} \sum_{Q_{k-1}\in \G_{k-1}}  
	 (f-e^{-\ell(Q_{k-1})^2L}f) \chi_{Q_{k-1}\backslash\cup_{Q_k\in\G_k} Q_k}\nonumber\\
	&\quad+\sum_{k= 1}^K \sum_{Q_k\in \G_{k}}
	 (e^{-\ell(Q_k)^2L}f-e^{-\ell({Q_k^\ast})^2L}f) \chi_{Q_k}\nonumber\\
&\quad+\sum_{Q_K\in \G_{K}} (f-e^{-\ell(Q_{K})^2L}f) \chi_{Q_{K}}.
\end{align*}
Note that when $K\to \infty$, $f-e^{-\ell(Q_{K})^2L}f\to 0$ and it follows from Lemma~\ref{C-k} below that
$\mu(\cup_{Q_K\in\G_K}Q_K)\to 0$.
Finally, it is seen that
\begin{align}\label{e3.5}
	(f-e^{-\ell(Q_0)^2L}f)\chi_{Q_0}&=\sum_{k\geq 1 } \sum_{Q_{k-1}\in \G_{k-1}}  
	 (f-e^{-\ell(Q_{k-1})^2L}f) \chi_{Q_{k-1}\backslash\cup_{Q_k\in\G_k} Q_k}\nonumber\\
	&\quad+\sum_{k\geq 1} \sum_{Q_k\in \G_{k}}
	 (e^{-\ell(Q_k)^2L}f-e^{-\ell({Q_k^\ast})^2L}f) \chi_{Q_k}.
\end{align}
We further investigate several fundamental properties with respect to the cubes $Q_k\in \G_k$ as in the decomposition \eqref{e3.5} above, which consist of Lemmas  \ref{A-k}--\ref{D-k}.

\begin{lemma}\label{A-k}
For every $Q_k$,
$$
\lambda<\frac{1}{\mu(2Q_k)}\int_{2Q_k}|f(y)-e^{-\ell({Q_k^\ast})^2L}f(y)|\,d\mu(y)\leq 2^n \lambda.
$$
\end{lemma}
\begin{proof}
	The left-hand inequality is due to the criterion~\eqref{eq3.1}. To  prove the right-hand inequality,
we denote by $\widetilde{Q_k}$  the father of $Q_k$. From our selection of cubes,
\begin{align*}
&\frac{1}{\mu(2Q_k)}\int_{2Q_k}|f(y)-e^{-\ell({Q_k^\ast})^2L}f(y)|d\mu(y)\\
&\quad\quad\leq \frac{\mu(2\widetilde{Q_k})}{\mu(2Q_k)}\frac{1}{\mu(2\widetilde{Q_k})}
\int_{2\widetilde{Q_k}}|f(y)-e^{-\ell({Q_k^\ast})^2L}f(y)|d\mu(y)
<2^n\lambda.
\end{align*}
This completes the proof.
\end{proof}

\begin{lemma}\label{B-k}
For every $Q_k$ and almost every $x\in 2Q_k$,
$$
|e^{-\ell(Q_k)^2L}f(x)-e^{-\ell({Q_k^\ast})^2L}f(x)|\leq C2^{2n}\lambda.
$$
\end{lemma}
\begin{proof}
For every $Q_k$, by ~\eqref{bmo1},
\begin{align*}
&|e^{-\ell(Q_k)^2L}f(x)-e^{-\ell({Q_k^\ast})^2L}f(x)|\\
&\leq |e^{-\ell(Q_k)^2L}(f-e^{-\ell({Q_k^\ast})^2L}f)(x)|+|e^{-(\ell(Q_k)^2+\ell({Q_k^\ast})^2)L}f(x)-e^{-\ell({Q_k^\ast})^2L}f(x)|\\
&\leq |e^{-\ell(Q_k)^2L}(f-e^{-\ell({Q_k^\ast})^2L}f)(x)|+\|f\|_{\bmo}.
\end{align*}
Then we estimate $|e^{-\ell(Q_k)^2L}(f-e^{-\ell({Q_k^\ast})^2L}f)(x)|$ with $x\in 2Q_k$. Denote by $\widetilde {Q_k}^{i}$  the $i$-th ancestor of $Q_k$ and let $i_0$ be the number such that $\widetilde {Q_k}^{i_0}$ is one of the children of ${Q_k^\ast}$.
We write
\begin{align}\label{eq:B-1}
 |e^{-\ell(Q_k)^2L}(f-e^{-\ell({Q_k^\ast})^2L}f)(x)|
&\leq \int_{X} |h_{\ell(Q_k)^2}(x,y)|\ |f(y)-e^{-\ell({Q_k^\ast})^2L}f(y)|d\mu(y)\nonumber\\
&\leq \int_{2Q_k} |h_{\ell(Q_k)^2}(x,y)|\ |f(y)-e^{-\ell({Q_k^\ast})^2L}f(y)|d\mu(y)\nonumber\\
&\quad+\int_{2\widetilde {Q_k}^{i_0}\backslash 2Q_k} |h_{\ell(Q_k)^2}(x,y)|\ |f(y)-e^{-\ell({Q_k^\ast})^2L}f(y)|d\mu(y)\nonumber\\
&\quad+\int_{X\backslash 2{\widetilde {Q_k}^{i_0}}} |h_{\ell(Q_k)^2}(x,y)|\ |f(y)-e^{-\ell({Q_k^\ast})^2L}f(y)|d\mu(y)\nonumber\\
&=:{\rm I} + {\rm II} + {\rm III}.
\end{align}
For the term ${\rm I}$,  from the pointwise upper bound of the heat kernel $h_{\ell(Q_k)^2}(x,y)$ and from Lemma~\ref{A-k}, we have
\begin{align*}
{\rm I}&\leq \frac{1}{\mu(Q_k)} \int_{2Q_k} |f(y)-e^{-\ell({Q_k^\ast})^2L}f(y)|d\mu(y)\\
&\leq \frac{2^n}{\mu(2Q_k)} \int_{2Q_k} |f(y)-e^{-\ell({Q_k^\ast})^2L}f(y)|d\mu(y)\\
&\leq2^{2n} \lambda.
\end{align*}
For the term ${\rm II}$,  we consider the chain of the dyadic cubes $\Big\{\widetilde {Q_k}^{i} \Big\}_{i=1}^{i_0}$ subject to the partial order via inclusion ``$\subseteq$'',  with the initial
dyadic cube $\widetilde {Q_k}^{1} $ which is the father of  $Q_k$  and the terminal one $\widetilde {Q_k}^{i_0} $ which is a child of ${Q_k^\ast}$.
From the pointwise upper bound  \textbf{(H2)} of the heat kernel $h_{\ell(Q_k)^2}(x,y)$,  we get
\begin{align*}
{\rm II}&\leq \sum_{i=1}^{i_0}\frac{1}{\mu(Q_k)}
e^{-c2^i}\int_{2\widetilde {Q_k}^{i}\backslash 2\widetilde {Q_k}^{i-1}} |f(y)-e^{-\ell({Q_k^\ast})^2L}f(y)|d\mu(y)\\
&\leq C\sum_{i=1}^{i_0}\frac{\mu(2\widetilde {Q_k}^{i})}{\mu(Q_k)}
e^{-c2^i}\frac{1}{\mu(2\widetilde {Q_k}^{i})}\int_{2\widetilde {Q_k}^{i}} |f(y)-e^{-\ell({Q_k^\ast})^2L}f(y)|d\mu(y).
\end{align*}
Then from our selection of cubes, we further have
\begin{align*}
{\rm II}\leq  C\lambda\sum_{i=1}^{i_0}2^{ni} e^{-c2^i}
\leq C\lambda.
\end{align*}
We now consider the term ${\rm III}$.
Note that in this case $d(x,y)\geq 2^{i_0}\ell(Q_k)=\ell({Q_k^\ast})$
for any $x\in 2Q_k $ and for $y\in M\backslash 2\widetilde {Q_k}^{i_0}$.
Then  we see that for any $x\in 2Q_k $ and for any $y\in M\backslash 2\widetilde {Q_k}^{i_0}$,
\begin{align*}
|h_{\ell(Q_k)^2}(x,y)|&\leq C\frac{1}{\mu(Q_k)} e^{-\frac{d(x,y)^2}{\ell(Q_k)^2}}\\
&\leq  C\frac{\mu({Q_k^\ast})}{\mu(Q_k)} \frac{1}{\mu({Q_k^\ast})}e^{-\frac{d(x,y)^2}{\ell({Q_k^\ast})^2}} e^{-2^{i_0}/2}\\
&\leq  C\frac{1}{\mu({Q_k^\ast})}e^{-\frac{d(x,y)^2}{\ell({Q_k^\ast})^2}}.
\end{align*}
Observe the fact that
$$
|e^{-2^{2i}\ell({Q_k^\ast})^2L}f-e^{-\ell({Q_k^\ast})^2L}f|\leq C  i \|f\|_{\bmo}, \ \ \ i=1, \ldots, i_0.
$$
A similar argument  as in the proof of ${\rm II}$ shows  that
${\rm III}\leq C \|f\|_{\bmo}\leq C\lambda$.

Combining the estimates of the terms ${\rm I}$, ${\rm II}$ and ${\rm III}$, we complete the proof of Lemma~\ref{B-k}.
\end{proof}

\begin{lemma}\label{C-k}
For arbitrary $\varepsilon'>\varepsilon$ with $\varepsilon$ in the John--Nirenberg inequality~\eqref{e1.11},
$$
\sum_{Q_k\in \G_k,Q_k\subset Q_k^\ast} \mu(Q_k)\leq Ce^{-\lambda/\varepsilon'}\mu({Q_k^\ast}).
$$
\end{lemma}
\begin{proof}
Recall the John--Nirenberg inequality that
$$
\mu(\{x\in Q: |f(x)-e^{-\ell(Q)^2L}f(x)|>\lambda\})\leq Ce^{-\lambda/\varepsilon}\mu(Q).
$$
If a set $E\subset Q$ satisfies
$$
\frac{1}{\mu(E)}\int_E |f(x)-e^{-\ell(Q)^2L}f(x)| \, d\mu(x)> \lambda,
$$
then we claim that for any $\varepsilon'>\varepsilon$ we have
\begin{align}\label{claim:JN}
\mu(E)\leq C_{\varepsilon'}e^{-\lambda/\varepsilon'}\mu(Q).
\end{align}
Actually Jensen's inequailty gives
$$
e^{\lambda/\varepsilon'}\leq e^{\frac{1}{\mu(E)}\int_E |f(x)-e^{-\ell(Q)^2L}f(x)|\, d\mu/\varepsilon'}\leq \frac{1}{\mu(E)}\int_E e^{|f(x)-e^{-\ell(Q)^2L}f(x)|/\varepsilon'} \, d\mu(x).
$$
And then
\begin{align*}
\mu(E)&\leq e^{-\lambda/\varepsilon'} \int_E e^{|f(x)-e^{-\ell(Q)^2L}f(x)|/\varepsilon'} d\mu(x)\\
&\leq e^{-\lambda/\varepsilon'} \int_Q e^{|f(x)-e^{-\ell(Q)^2L}f(x)|/\varepsilon'} d\mu(x)\\
&\leq e^{-\lambda/\varepsilon'} \left(\int_{|f-e^{-\ell(Q)^2L}f|\leq \lambda_0} e^{|f(x)-e^{-\ell(Q)^2L}f(x)|/\varepsilon'} d\mu(x)\right.\\
&\left.\qquad\qquad\qquad+ \sum_{k\geq 1}\int_{k\lambda_0<|f-e^{-\ell(Q)^2L}f|\leq (k+1)\lambda_0} e^{|f(x)-e^{-\ell(Q)^2L}f(x)|/\varepsilon'} d\mu(x)\right)\\
&\leq e^{-\lambda/\varepsilon'}\left( e^{\lambda_0/\varepsilon'}\mu(Q)+\sum_{k\geq 1}e^{(k+1)\lambda_0/\varepsilon'}\mu\big(\{x\in Q: |f(x)-e^{-\ell(Q)^2L}f(x)|>k\lambda_0\}\big)\right)\\
&\leq e^{-\lambda/\varepsilon'}\left( e^{\lambda_0/\varepsilon}\mu(Q)+\sum_{k\geq 1}e^{(k+1)\lambda_0/\varepsilon'}e^{-k\lambda_0/\varepsilon}\mu(Q)\right)\\
&\leq C_{\varepsilon,\varepsilon',\lambda_0} e^{-\lambda/\varepsilon'}\mu(Q).
\end{align*}
It follows from Vitali Covering Lemma~\cite[p. 9]{St3} that we can choose disjoint subsequce of $\{2Q_{k_i}\}$ from $\{2Q_k\}$ where
 $Q_k\in \G_k,Q_k\subset Q_k^\ast$ such that
$$
\sum_{Q_k\in \G_k,Q_k\subset Q_k^\ast} \mu(Q_k)\leq C \sum_{i} \mu(2Q_{k_i})=C\mu\Bigg(\bigcup_{i}2Q_{k_i}\Bigg),
$$
where $C$ depends only the doubling constants in \eqref{doubling2}.
For every ball $2Q_{k_i}$, from our selection criterion~\eqref{eq3.1} we have
$$
\frac{1}{\mu(2Q_{k_i})}\int_{2Q_{k_i}} |f(x)-e^{-\ell({Q_k^\ast})^2L}f(x)| \, d\mu(x)> \lambda
$$
and thus
$$
\frac{1}{\mu(\cup_{i}2Q_{k_i})}\int_{\bigcup_{i}2Q_{k_i}} |f(x)-e^{-\ell({Q_k^\ast})^2L}f(x)| \, d\mu(x)> \lambda.
$$
Then from our claim~\eqref{claim:JN}, we have
$$\sum_{Q_k\in \G_k,Q_k\subset Q_k^\ast} \mu(Q_k)\leq C_{n,\varepsilon'}e^{-\lambda/\varepsilon'}\mu({Q_k^\ast}).$$
The proof is complete.
\end{proof}

\begin{lemma}\label{D-k}
For almost every $x$ on the set ${Q_k^\ast}\backslash\cup_{Q_k\in \G_k,Q_k\subset Q_k^\ast} Q_k$,
$$
|f(x)-e^{-\ell({Q_k^\ast})^2L}f(x)|\leq \lambda.
$$
\end{lemma}
\begin{proof}
The criterion~\eqref{eq3.1} gives that for $x\in {Q_k^\ast}\backslash \cup Q_k$, we have that
$$
|f(x)-e^{-\ell({Q_k^\ast})^2L}f(x)|\leq \overline{\lim_{\ell(R)\to 0,x\in 2R}}\frac{1}{\mu(2R)}\int_{2R}|f(y)-e^{-\ell({Q_k^\ast})^2L}f(y)|d\mu(y)\leq \lambda,
$$
where $R$ are the descendants of ${Q_k^\ast}$ which contain $x$ and are not selected.
\end{proof}

 \medskip

\noindent
{\bf Back to the proof of Proposition~\ref{lem3.2}}  From \eqref{e3.5},   we rewrite the function $(f-e^{-\ell(Q_0)^2L}f)$ in the following:
\begin{eqnarray} \label{e3.8}
(f-e^{-\ell(Q_0)^2L}f)\chi_{Q_0} = g_{Q_0}+h_{Q_0},
\end{eqnarray}
where
$$
g_{Q_0}=\sum_{k\geq 1 } \sum_{Q_{k-1}\in \G_{k-1}}
	 (f-e^{-\ell(Q_{k-1})^2L}f) \chi_{Q_{k-1}\backslash\cup_{Q_k\in\G_k} Q_k}
	$$
and
$$
h_{Q_0}= \sum_{k\geq 1} \sum_{Q_k\in \G_{k}}
		 (e^{-\ell(Q_k)^2L}f-e^{-\ell({Q_k^\ast})^2L}f) \chi_{Q_k}
$$

 From Lemma~\ref{D-k}, it is seen that $\|g_{Q_0}\|_\infty\leq C\lambda$. From Lemma~\ref{C-k}, we know that $$
 \left\{E_k:=\cup_{Q_k\in \G_k}Q_k\right\}_{k}
 $$
  are $2^{-\lambda/\varepsilon'}$-graded sequence.
Denote by
	\begin{eqnarray}\label{aaab}
 a_{Q_k}(x)=e^{-\ell(Q_k)^2L}f(x)-e^{-\ell({Q_k^\ast})^2L}f(x).
	\end{eqnarray}
Lemma~\ref{B-k} shows that $ a_{Q_k}(x)$ satisfies condition~\eqref{size condition3.1}.

The proof of  the estimation  \eqref{smooth condition3.3}  of  $ a_{Q_k}(x)$  is a consequence of
  the following lemma. Once it is proved,    the proof of Proposition~\ref{lem3.2} is complete.
\end{proof}

\begin{lemma}\label{lem3.6}
	If $f\in \bmo$, then for every $t>0$ and every $K\geq 1$, there exists a constant
	$C>0$ such that for almost all $x\in X$, we have
	\begin{eqnarray}\label{bmo2}
		 |(tL)(e^{-tL}-e^{-KtL})f(x)|\leq C\|f\|_{\bmo}.
	\end{eqnarray}
\end{lemma}
\begin{proof} The proof of \eqref{bmo2} is based on the following fact that
 the Gaussian upper bounds for $p_t(x,y)$ are further
	inherited by the time derivatives of $p_t(x,y)$. That is, for each
	$k\in{\mathbb N}$, there exist two positive constants $c_k$ and $C_k$ such
	that	
	\begin{eqnarray}\label{e3.10}
		\Big|{\partial^k \over\partial t^k} p_t(x,y) \Big|\leq
		{C_k\over  t^k\mu(B(x, \sqrt{t}))}\exp\Big(-\frac{d^2(x,y)}{c_kt}\Big),\
		\quad\forall \,t>0,
	\end{eqnarray}
	  for almost every $x,y\in X$ (see  for example,    \cite[Theorem~6.17]{Ou}).

	  Now for any $K>1$, let $\ell$ be an integer such that $2^\ell \leq K< 2^{\ell+1}$. One can write
	  \begin{align*}
	  	 |(tL)(e^{-tL}-e^{-KtL})f(x)|&\leq \sum_{k=0}^{\ell-1} 2^{-k}|(2^ktL)(e^{-2^ktL}-e^{-2^{k+1}tL})f(x)|\\
	  	&\qquad+2^{-\ell}|(2^\ell tL)(e^{-2^\ell tL}-e^{-KtL})f(x)|.
	  \end{align*}
From \eqref{bmo1} and \eqref{e3.10}, we see  that  
   	 \begin{align*}
   		 |(tL)(e^{-tL}-e^{-KtL})f(x)|
   		&\leq C\|f\|_{\bmo} + 2^{-\ell}|(2^\ell tL)(e^{-2^\ell tL}-e^{-KtL})f(x)|,\,\,{\rm a.e.}\,x\in X.
   	\end{align*}
	An  the argument as in \cite[p. 1386--1387]{DY1} shows that
	$
	|(2^\ell tL)(e^{-2^\ell tL}-e^{-KtL})f(x)|\leq C\|f\|_{\bmo},
	$
	and so 	$	 |(tL)(e^{-tL}-e^{-KtL})f(x)| \leq C\|f\|_{\bmo}$. This proves \eqref{bmo2}, and completes
	 the proof of Lemma~\ref{lem3.6}.
\end{proof}

\medskip

 Next we state the following properties of the  family of functions $\{ a_{Q_k}(x)\}$, which will be useful in Section 5 below.

\begin{lemma}\label{lem3.7}
	 Let  $\{a_{Q_k}\} $ be the  family of functions given  in \eqref{e3.5}.
Then for almost all $x\in 2Q_k$ and $0<t\leq \ell(Q_k),$
	\begin{align}\label{smooth condition3.2}
		 |(I-e^{-t^2L})a_{Q_k}(x)|\leq C \left(\frac{t}{\ell(Q_k)}\right)^2\lambda
	\end{align}
	and
	\begin{align}\label{support condition3.4}
		\int_X |h_{t^2}(x,y)|\ |(\ell(Q_k)^2L)^m a_{Q_k}(y)|d\mu(y)\leq C\lambda, \ \ \ \ \  m=0, 1.
	\end{align}
	\end{lemma}

	\begin{proof}
	We write
		$
		 1-e^{-t^2\lambda}=\int_0^{t^2} \lambda e^{-s\lambda} ds.
		$
	It follows from \eqref{bmo2} that
		\begin{align*}
			 |(I-e^{-t^2L})a_{Q_k}(x)|&\leq [{\ell(Q_k)]^{-2}}\int_0^{t^2} |e^{-sL} (\ell(Q_k)^2L)(e^{-\ell(Q_k)^2L}-e^{-\ell({Q_k^\ast})^2L})f(x)|\, ds\\
			&\leq C[{\ell(Q_k)]^{-2}}\int_0^{t^2} \|(\ell(Q_k)^2L)(e^{-\ell(Q_k)^2L}-e^{-\ell({Q_k^\ast})^2L})f\|_{L^\infty} \, ds\\
			&\leq C t^2[{\ell(Q_k)]^{-2}}\|f\|_{\bmo}\\
			&\leq C t^2\lambda[{\ell(Q_k)]^{-2}}.
		\end{align*}
		This completes the proof of  \eqref{smooth condition3.2}.
		
		We now prove \eqref{support condition3.4}.  The case $m=1$ is a consequence of ~\eqref{bmo2} and the heat kernel $h_t(x,y)\in L^1$ uniformly for $t$. For $m=0$, one can write
		\begin{align*}
			 e^{-\ell(Q_k)^2L}f(x)-e^{-\ell({Q_k^\ast})^2L}f(x)&=(e^{-\ell(Q_k)^2L}-
			 e^{-(\ell({Q_k^\ast})^2+\ell(Q_k)^2)L})f(x)\\
			 &\quad+(e^{-(\ell({Q_k^\ast})^2+\ell(Q_k)^2)L}-e^{-\ell({Q_k^\ast})^2L})f(x).
		\end{align*}
		It follows from ~\eqref{bmo1} that
		$
		 |(e^{-\ell({Q_k^\ast})^2L}-e^{-(\ell({Q_k^\ast})^2+\ell(Q_k)^2)L})f(x)|\leq C\|f\|_{\bmo}.
		$
	A similar argument as in estimate~\eqref{eq:B-1} in Lemma~\ref{B-k} shows  that
$$
		\int_X |h_{t^2}(x,y)|\ |e^{-\ell(Q_k)^2L}(I-e^{-\ell({Q_k^\ast})^2L})f(y)| \, d\mu(y)\leq C \lambda,
		$$
		and   completes the proof of \eqref{support condition3.4}. The proof of Lemma~\ref{lem3.7} is complete.
	\end{proof}

\medskip

\section{Proof of Theorem~\ref{th1.1}: step two}
\setcounter{equation}{0}

Theorem~\ref{th1.1} consists of fixing $\varepsilon>\varepsilon_L (f)$, choosing a large $\lambda>\lambda_0(\varepsilon,f)$, and constructing $
h\in {\rm BMO}_L(X)$ in \eqref{e1.13} so that $\|h\|_{{\rm BMO}_L(X)}\leq c\varepsilon$
and that $\|f-h\|_{L^{\infty}(X)}\leq C\lambda$. That is, in the construction of $h$, the term $\|h\|_{{\rm BMO}_L(X)}$ should be small enough.
To do so, we continue the following construction of ${\rm BMO}_L(X)$ functions which are supported on a given cube $Q_0$.

\begin{proposition}\label{mainth2}Suppose that $L$ is a densely-defined operator
	on $L^2(X)$ satisfying \textbf{(H1)} and \textbf{(H2)}.
	Fixed a dyadic system $\mathcal D$ and large dyadic cube $Q_0\in \mathcal D$. There are constants $A_1$ and $A_2$, depending only on the doubling constants in \eqref{doubling2} and the Gaussian upper bound, such that if $f\in \BL(X)$ and if $\varepsilon>\varepsilon_L(f)$, then
	$$
	 \big(f(x)-e^{-\ell(Q_0)^2L}f(x)\big)\chi_{Q_0}(x)=g_{Q_0}(x)+h_{Q_0}(x)
	$$
	with
	$$
	\|g\|_{L^\infty}\leq A_1\lambda
	$$
	and
	$$
	h_{Q_0}(x)= {\varepsilon\over \lambda} \sum_Q a_{Q}(x)\chi_{Q}(x),
	$$
	where the cubes $\{Q\}$ run over a $\gamma$-graded sequence with $0<\gamma<1$ and $a_{Q}(x)$ satisfies that for almost all $x\in 2Q$ and $0<t\leq \ell({Q})$ ,
	\begin{align}\label{size condition}
		|a_{Q}(x)|\leq C\lambda,
	\end{align}
	and
	\begin{align}\label{smooth condition1}
		|La_Q(x)|\leq C\ell({Q})^{-2}\lambda.
	\end{align}
\end{proposition}

 To begin with, we first adjust the graded sequence $\{E_k:=\cup_{Q_k\in\G_k}Q_k\}_{k}$ by adding more cubes into it.
 Recall $\G_k$ is defined by \eqref{Gk} and ${Q_k^\ast}\in \G_{k-1}$ is the dyadic cube containing $Q_k\in \G_k$.
  We need the following lemma.

\begin{lemma}\label{lem3.3}
For $2^{-\lambda/\varepsilon'}$-graded sequence $\{E_k:=\cup_{Q_k\in\G_k}Q_k\}_{k}$, we can  construct a new $2^{-1}$-graded sequence $\{E_{k,i}\}$, $0\leq i\leq m,k\geq 1$,  such that
$$
E_k\subset E_{k,m-1}\subset E_{k,m-2}\subset \ldots\subset E_{k,1}\subset E_{k-1},
$$
 where we denote by $E_{k,0}=E_{k-1}$ and $E_{k,m}=E_k$.
\end{lemma}

\begin{proof}
Let $m$ be the largest integer that is smaller than $\lambda/((n+1)\varepsilon')$.
The first step is that we construct $2^{-1}$-graded sequence between $E_k$ and $E_{k-1}$, that is,
$$
E_k\subset E_{k,m-1}\subset E_{k,m-2}\subset \ldots\subset E_{k,1}\subset E_{k-1}.
$$
For each $Q_k\in \G_k$, choose the smallest $Q_k'$ such that $Q_k'$ is an ancestor of $Q_k$ and satisfies
$$
2^{-n-1} \mu(Q_k')\leq \mu(E_k\cap Q_k')\leq 2^{-1} \mu(Q_k').
$$
Now consider the family set of $\{Q_k'\}$ constructed from all $Q_k\in \G_k$. Delete those cubes  which are from $\{Q_k'\}$ and contained in another cube from $\{Q_k'\}$. Denote by $\mathcal G_{k,m-1}$ the set of $\{Q_k'\}$ after deletion. Then
$$
E_{k,m-1}:=\bigcup_{Q\in \G_{k,m-1}} Q.
$$
Iterating the above process, we get $E_{k,m-2}$ from $E_{k,m-1}$ and all other $E_{k, i}$  and the corresponding sets of cubes $\mathcal G_{k,i}$, $i=1,2,\ldots,m-1$. These $E_{k,i}$ are $2^{-1}$-graded sequence automatically from the construction. What remains to show is that it is still $2^{-1}$-graded from $E_{k,1}$ to $E_{k-1}$, that is,
$$
\mu(E_{k,1}\cap {Q_k^\ast})\leq 2^{-1} \mu({Q_k^\ast}).
$$
In fact, since $E_k$ is $2^{-\lambda/\varepsilon'}$-graded, that is,
$$
\mu(E_{k}\cap {Q_k^\ast})\leq 2^{-\lambda/\varepsilon'} \mu({Q_k^\ast}),
$$
we obtain that
\begin{align*}
\mu(E_{k,m-1}\cap {Q_k^\ast})=\sum_{Q_k'\subset {Q_k^\ast}} \mu(Q_k')\leq \sum_{Q_k'\subset {Q_k^\ast}}  2^{n+1}\mu(E_k\cap Q_k')
\leq 2^{n+1} \mu({Q_k^\ast}\cap E_k).
\end{align*}
Inductively, we have
$$
\mu(E_{k,1}\cap {Q_k^\ast})\leq 2^{(n+1)(m-1)} \mu({Q_k^\ast}\cap E_k)\leq  2^{(n+1)(m-1)-\lambda/\varepsilon'}\mu({Q_k^\ast})\leq 2^{-1}\mu({Q_k^\ast}).
$$
The proof is complete.
\end{proof}

\medskip

We now  prove  Proposition~\ref{mainth2}. With the new $2^{-1}$-graded sequence $\{E_{k,i}\}$, $0\leq i\leq m,k\geq 1$,
based on the decomposition in Proposition~\ref{lem3.2}, we will add   more $L^\infty$ functions in   $h_{Q_0}(x)=\sum_Q a_Q(x)\chi_Q(x)$  to lower the oscillation. However, before we do that, because $a_Q(x)$ is not a constant, we should rewrite the decomposition in Proposition~\ref{lem3.2} by some kind of average of $a_Q(x)$.

\begin{proof}[Proof of Proposition~\ref{mainth2}]
It follows from  Proposition~\ref{lem3.2} that we have a decomposition of function $(f-e^{-\ell(Q_0)^2L}f)\chi_{Q_0}=g_{Q_0}+h_{Q_0}$
where $h_{Q_0}$ is a $2^{-\lambda/\varepsilon'}$-graded sequence function. Based on this decomposition and the new $2^{-1}$-graded sequence of $\{E_{k,i}\}$ constructed in Lemma~\ref{lem3.3}, we construct a new decomposition of function $(f-e^{-\ell(Q_0)^2L}f)\chi_{Q_0}$.
Let $m$ be the largest integer that is smaller than $\lambda/((n+1)\varepsilon')$. Denote by $Q_{k,i}$ the related cubes in $E_{k,i}$ such that $Q_k\subset Q_{k,i}$. And denote by $Q_{k,i}^\ast$ the cube $Q_{k-1,i}$ such that $Q_{k,i}\subset Q_{k-1}$.

Define
$$
Q_{0}^{i}:=Q_0,\quad\quad 0\leq i\leq m
$$
and
$$
 E_{Q_k,L}:=\frac{1}{m}\sum_{i=0}^m e^{-\ell(Q_{k,i})^2L}, \quad\quad k\geq 0.
$$

Then we replace $e^{-\ell(Q_k)^2L}$ in \eqref{e3.5} by $E_{Q_k,L}$ to get
\begin{align}
 (f-e^{-\ell(Q_0)^2L}f)\chi_{Q_0} =:g''+g'''+h\nonumber,
\end{align}
where
\begin{eqnarray*}
g''&=&\sum_{k\geq 1 } \sum_{Q_{k-1}\in \G_{k-1}}
	 (f-E_{Q_{k-1},L}f) \chi_{Q_{k-1}\backslash\cup_{Q_k\in\G_k} Q_k},\\
g'''&=&\sum_{k\geq 1} \sum_{Q_{k}\in \G_{k}} (E_{Q_k,L}f-E_{{Q_k^\ast},L}f) \chi_{Q_k}\\
& &-\sum_{k\geq 1}\sum_{0\leq i\leq m} \sum_{Q_{k,i}\in \G_{k,i}} \frac{e^{-\ell(Q_{k,i})^2L}f-e^{-\ell({Q_{k,i}^\ast})^2L}f}{m}\chi_{Q_{k,i}}
\end{eqnarray*}
and
$$
h=\frac{1}{m}\sum_{k\geq 1} \sum_{0\leq i\leq m} \sum_{Q_{k,i}\in \G_{k,i}} \big({e^{-\ell(Q_{k,i})^2L}f-e^{-\ell({Q_{k,i}^\ast})^2L}f}\big)\chi_{Q_{k,i}}.
$$

Next we shall prove $\|g''\|_\infty\leq C\lambda$ and $\|g'''\|_\infty\leq C\lambda$. Before the proof, we note that the cube $Q_{k,i}$, $i\neq 0, m$, is an ancestor of some cube $Q_k$. So $Q_{k,i}$ does not satisfy the
selection criterion, nor the  ancestors of $Q_{k,i}$ except ${Q_k^\ast}$. Thus it follows the similar argument as in the proof of Lemma~\ref{B-k} that for $x\in 2Q_{k,i}$,
\begin{eqnarray}\label{emm}
|e^{-\ell(Q_{k,i})^2L}f(x)-e^{-\ell({Q_k^\ast})^2L}f(x)|\leq C\lambda.
\end{eqnarray}
 For $g''$, note that the supports of every function in the summation of $g''$ do not intersect and it follows from Lemmas \ref{B-k} and \ref{D-k} and estimate \eqref{emm} that for $x\in Q_{k}\backslash\cup_{Q_{k+1}\in\G_{k+1}} Q_{k+1}$,
\begin{align*}
 |f(x)-E_{{Q_k^\ast},L}f(x)|
&\leq \frac{1}{m}\sum_{i=0}^m |f(x)-e^{-\ell(Q_{k,i})^2L}f(x)|\\
&\leq \frac{1}{m}\sum_{i=0}^m \Big\{|f(x)-e^{-\ell(Q_k)^2L}f(x)|+|e^{-\ell(Q_{k,i})^2L}f(x)-e^{-\ell({Q_k^\ast})^2L}f(x)|\\
&\hskip 4cm+|e^{-\ell(Q_k)^2L}f(x)-e^{-\ell({Q_k^\ast})^2L}f(x)|\Big\}\\
&\leq \frac{1}{m}\sum_{i=0}^m C2^n \lambda\\
& \leq C\lambda.
\end{align*}
For $g'''$, note that if $x\in Q_k$, then $x\in Q_{k,i}$ for all $i$, and thus
$$
\sum_{k\geq 1} \sum_{Q_{k}\in \G_{k}} (E_{Q_k,L}f-E_{{Q_k^\ast},L}f) \chi_{Q_k}
-\sum_{k\geq 1} \sum_{0\leq i\leq m} \sum_{Q_{k,i}\in \G_{k,i}}\frac{e^{-\ell(Q_{k,i})^2L}f-e^{-\ell({Q_{k,i}^\ast})^2L}f}{m} \chi_{\cup_{Q_{k}\subset Q_{k,i}}Q_k}=0.
$$
Note that for different $k$, $Q_{k,i}\backslash \cup_{Q_{k}\subset Q_{k,i}}Q_k$ does not intersect each other.
With a fixed $x$ and the same $k$, there are at most number $m$ of $Q_{k,i}$s such that $x\in Q_{k,i}$. Thus  it follows from  Lemmas \ref{B-k} and \ref{D-k} and estimate \eqref{emm} that
\begin{eqnarray}\label{eq:a_Q} \hspace{1cm}
&&\hspace{-1cm} \sum_{k\geq 1}  \sum_{0\leq i\leq m}\sum_{Q_{k,i}\in \G_{k,i}} \frac{e^{-\ell(Q_{k,i})^2L}f-e^{-\ell({Q_{k,i}^\ast})^2L}f}{m}\chi_{Q_{k,i}\backslash \cup_{Q_{k}\subset Q_{k,i}}Q_k}\nonumber\\
&\leq& \sup_{x\in Q_{k,i}\backslash \cup_{Q_{k}\subset Q_{k,i}}Q_k} |e^{-\ell(Q_{k,i})^2L}f(x)-e^{-\ell({Q_{k,i}^\ast})^2L}f(x)|\nonumber\\
&\leq& \sup_{x\in Q_{k,i}\backslash \cup_{Q_{k}\subset Q_{k,i}}Q_k} |e^{-\ell(Q_{k,i})^2L}f(x)-e^{-\ell({Q_k^\ast})^2L}f(x)|+|e^{-\ell({Q_k^\ast})^2L}f(x)-e^{-\ell(Q_k^{\ast\ast})^2L}f(x)|\nonumber\\
&&+
 |e^{-\ell({Q_{k,i}^\ast})^2L}f(x)-e^{-\ell(Q_k^{\ast\ast})^2L}f(x)|\nonumber\\
&\leq& C\lambda,
\end{eqnarray}
where $Q_k^{\ast\ast}=(Q_k^{\ast})^\ast\in \G_{k-2}$ such that $Q_k\subset Q_k^\ast\subset Q_k^{\ast\ast}$.
\medskip

Now we have
$$
\big(f(x)-e^{-\ell(Q_0)^2L}f(x)\big)\chi_{Q_0}(x)=g_{Q_0}(x)+h_{Q_0}(x)
$$
with
$g_{Q_0}=g''+g'''$ and $h_{Q_0}=h$. Based on the estimates of $g''$ and $g'''$, we have
$$
\|g_{Q_0}\|_{L^\infty}\leq A_1\lambda
$$
For $h_{Q_0}$, we write
$$
h_{Q_0}=h=\frac{1}{m}\sum_{k\geq 1} \sum_{0\leq i\leq m}\sum_{Q_{k,i}\in \G_{k,i}} \big(e^{-\ell(Q_{k,i})^2L}f-e^{-\ell({Q_{k,i}^\ast})^2L}f\big)\chi_{Q_{k,i}}=\frac{\varepsilon'}{\lambda} \sum_Q a_{Q}(x)\chi_{Q}(x),
$$
where $Q=Q_{k,i}$ runs over all the $2^{-1}$-graded sequence which we have already constructed and
$$
a_{Q}(x):=(n+1) \big(e^{-\ell(Q_{k,i})^2L}f(x)-e^{-\ell({Q_{k,i}^\ast})^2L}f(x)\big).
$$

A similar argument as in \eqref{eq:a_Q}  shows that  \eqref{size condition} holds. From \eqref{bmo2} of Lemma~\ref{lem3.6},   \eqref{smooth condition1} follows readily.
 The proof of Proposition~\ref{mainth2} is complete.
\end{proof}

\medskip
\section{Proof of Theorem~\ref{th1.1}: step three}
\setcounter{equation}{0}

In this section we  give
a decomposition   \eqref{e1.13} of    $f\in {\rm BMO}_L(X)$
  from  the dyadic result in Section 3.

\begin{theorem}\label{mainth1}
	Suppose $L$ is an operator
	on $L^2(X)$ satisfying \textbf{(H1)} and \textbf{(H2)}.
	There are constants $A_1$ and $A_2$, depending only on the doubling constants as in \eqref{doubling2} and the Gaussian upper bound, such that if $f\in \BL(X)$ and if $\varepsilon>\varepsilon_L(f)$, then
	$$
	f(x)=g(x)+h(x)
	$$
	with
	$$
	\|g\|_{L^\infty}\leq A_1\lambda_0
	$$
	and
\begin{eqnarray}\label{4.0}
	\|h\|_{\BL(X)}\leq A_2 \varepsilon.
\end{eqnarray}
\end{theorem}

In order to estimate the $\BL(X)$ norm, we will use
the fact that
$\BL(X)$ is the dual of the Hardy space $ H_L^1(X)$ associated with the operator $L$, which was  proved in \cite{DY2}. That is,
$$
(H^1_{L}(X))^{\ast}={{\rm BMO}_L(X)}.
$$
Recall that
the quadratic operators associated
with $L$ is defined by
\begin{eqnarray*}
S_{L}f(x)=\left(\int_0^{\infty}\!\!\!\!\int_{\substack{  d(x,y)<t}}
|(t^2L) e^{-t^2L} f(y)|^2 {d\mu(y)\over \mu(B(x,t))}{dt\over t}\right)^{1/2},
\quad x\in X
\end{eqnarray*}
where $f\in L^2(X)$.
Define the spaces $ H_L^1(X)$ as the completion of $L^2(X)$ in
the norms given by the $L^1$ norm of the square function
$\|f\|_{H^1_{L}(X)}:=\|S_{L}f\|_{L^1(X)}$.
The  Hardy  space $ H_L^1(X)$  associated with $L$ was first
introduced by   Auscher, Duong and McIntosh \cite{ADM}, and then Duong and Yan  \cite{DY2},
introduced Hardy and ${\rm BMO}$ spaces explicitly adapted to an
operator $L$ whose heat kernel enjoys  a pointwise Gaussian upper bound (but see also the
earlier, more specific work of Auscher and Russ \cite{AR});
see
\cite{AMR, DY1, DY2, HLMMY, HM, HMM}
and the references therein. Atom decomposition of $ H_L^1(X)$ was introduced in \cite{HLMMY}.
The following
$(1,\infty, M)$-atom associated with
an operator $L$ was  established on $\R^n$ in \cite{SY1} and on space of homogenous type in \cite{SY2}.
	Given  $M\in {\mathbb N}$,
a function $a\in L^2(X)$ is called a $(1,\infty, M)$-atom associated with
the operator $L$ if there exist a function $b\in {\mathcal D}(L^M)$
and a ball $B\subset X$ such that

\smallskip

{  (i)}\ $a=L^M b$;

\smallskip

{  (ii)}\ {\rm supp}\  $L^{k}b\subset B, \ k=0, 1, \dots, M$;

\smallskip

{  (iii)}\ $\|(r_B^2L)^{k}b\|_{L^{\infty} (X)}\leq r_B^{2M}
{\mu(B)}^{-1},\ k=0,1,\dots,M$.


\begin{lemma} \label{def1.3}
	 Suppose $L$ is an operator
	on $L^2(X)$ satisfying \textbf{(H1)} and \textbf{(H2)}.
	For every $M\in{\mathbb N}$,
	   $f\in H^1_L(X)$ if and only if  there exist
	$\lambda_j\in{\mathbb R}$ and  $(1, \infty, M)$-atom $a_j, j=1, 2, \cdots, $ such that	
  $$
  f= \sum\limits_{j=0}^{\infty}\lambda_j
	a_j \ \ \ {\rm  in } \ L^1(X)
	$$
	and
	$$
	c\|f\|_{H^1_L(X)}\leq
\sum_{j=0}^{\infty} |\lambda_j|\leq 	C \|f\|_{H^1_L(X)}.
	$$
\end{lemma}

\begin{proof}
For the proof, we refer the reader to  \cite[Theorem 2.5]{HLMMY}.
 \end{proof}

 Now we  prove Theorem~\ref{mainth1}.

\begin{proof}[Proof of Theorem~\ref{mainth1}]
Now for dyadic system $\omega\in \Omega$, denote by $Q_\alpha^K(\omega)\in \mathcal D^K(\omega)$ the dyadic cube with level $K$ in this dyadic system.
 Note that $\cup_\alpha Q_{\alpha}^K(\omega)=X$. $f$ can be decomposed as
 $$
 f(x)=\sum_\alpha f(x)\chi_{Q_\alpha^K(\omega)}(x).
 $$
In each $Q_{\alpha}^K(\omega)$,
it follows from Theorem~\ref{mainth2} that we have a decomposition of $f$ as
$$
\left(f(x)-e^{-\ell(Q_{\alpha}^K(\omega))^2L}f(x)\right)\chi_{Q_\alpha^K(\omega)}(x)=
g_{Q_{\alpha}^K(\omega)}(x)+h_{Q_{\alpha}^K(\omega)}(x).
$$

Then we have
\begin{align}\label{e4.0}
f=\sum_{\alpha} \left(e^{-\ell(Q_{\alpha}^K(\omega))^2L}f\chi_{Q_{\alpha}^K(\omega)}+
g_{Q_{\alpha}^K(\omega)}\chi_{Q_{\alpha}^K(\omega)}+h_{Q_{\alpha}^K(\omega)}\chi_{Q_{\alpha}^K(\omega)} \right).
\end{align}
 Here $\|g_{Q_{\alpha}^K(\omega)}\|_\infty<C \lambda$ and $h_{Q_{\alpha}^K(\omega)}=\frac{\varepsilon}{\lambda}\sum_{Q^\omega} a_{Q^\omega}\chi_{Q^\omega}$, where $a_{Q^\omega}$ and $Q^\omega$ are as in Theorem~\ref{mainth2}.  Recall that $Q^\omega$ is in the $2^{-1}$-graded sequence.
We also recall the properties of function $a_{Q^\omega}(x)$:

i) size condition: for almost all $x\in 2Q^\omega$,
\begin{align}\label{size condition2}
|a_{Q^\omega}(x)|\leq C\lambda;
\end{align}

ii) smooth condition: for almost all $x\in 2Q^\omega$,
\begin{align}\label{smooth condition}
|La_{Q^\omega}(x)|\leq C\ell({Q^\omega})^{-2}\lambda,
\end{align}
where $C$ just depends on $n$ in the doubling condition~\eqref{doubling2}.

Since $\sum_{\alpha}g_{Q_{\alpha}^K(\omega)}\chi_{Q_{\alpha}^K(\omega)}$ is uniformly bounded in $L^\infty(X)$, we have weak $\ast$ type limit $g^\omega$ for a subsequence of $K\to \infty$. For simplicity, we also denote by $K$ this subsequence. Define $h^\omega:=f-g^\omega$. Then
$$
f=g+h:=\E_\omega g^\omega+\E_\omega h^\omega.
$$
Note that $g\in L^\infty(X)\subset \bmo$ and $f\in \bmo$ imply that $h\in \bmo$. So what we need to prove is $\|h\|_{\bmo}\leq C\varepsilon$ and equivalently by duality  $(H_L^1(X))^*=\bmo$
$$
|\langle h,a\rangle|\leq C\varepsilon
$$
uniformly for all $(1,\infty, 2)$ atoms $a$.
Denote by $h^\omega_K=\sum_\alpha h_{Q_{\alpha}^K(\omega)}\chi_{Q_{\alpha}^K(\omega)}$. It follows from the definition of $h$
and equality~\eqref{e4.0} that
\begin{align}\label{e4.2}
\langle (h-\E_\omega h^\omega_K),a\rangle&=  \langle (f-\E_\omega g^\omega-\E_\omega h^\omega_K),a\rangle\nonumber\\
&= \Big\langle \E_\omega \Big(\sum_{\alpha}
g_{Q_{\alpha}^K(\omega)}\chi_{Q_{\alpha}^K(\omega)}- g^\omega\Big),a\Big\rangle+ \Big\langle \E_\omega\Big(\sum_{\alpha}e^{-\ell(Q_{\alpha}^K(\omega))^2L}f\chi_{Q_{\alpha}^K(\omega)}\Big),a\Big\rangle
\nonumber\\
&=:{\rm I}+{\rm II}.
\end{align}
For the term ${\rm I}$, note that $g^\omega$ is the weak $\ast$ type limit of $\sum_{\alpha}g_{Q_{\alpha}^K(\omega)}\chi_{Q_{\alpha}^K(\omega)}$.
Thus for any $\varepsilon>0$, taking large enough $K$ gives
\begin{align*}
\left|\Big\langle \E_\omega\Big(\sum_{\alpha}
g_{Q_{\alpha}^K(\omega)}\chi_{Q_{\alpha}^K(\omega)}- g^\omega\Big),a\Big\rangle\right|&\leq  \E_\omega \left|\Big\langle
g_{Q_{\alpha}^K(\omega)}\chi_{Q_{\alpha}^K(\omega)}- g^\omega,a\Big\rangle\right|\\
&\leq \frac{1}{100}\varepsilon.
\end{align*}
Denote by $Q$ the support of the atom $a$. Set
\begin{eqnarray*}
    &&\La_k := \{\omega\in\Omega: \mbox{ there exists a cube }Q^\omega\in\D^{k+1}(\omega)
        \mbox{ such that }Q\subset Q^\omega\},\\
    &&\La := \bigcap_{k\in\Z}(\La_k)^c
        = \{\omega\in\Omega:\mbox{ there is no cube\ } Q^\omega\in\D(\omega)
        \mbox{\ that contains\ } Q   \}.
\end{eqnarray*}
Note that $(\La_k)^c\supset (\La_{k-1})^c$. By
condition~\eqref{Rcondition},
$$
\PP(\La)=\lim_{k\rightarrow -\infty}P((\La_k)^c)\leq\lim_{k\rightarrow -\infty}
\Big(\frac{\ell(Q)}{2^{-k}}\Big)^\eta=0.
$$
To handle the term ${\rm II}$, we will fix the value of $\widehat{K}=\widehat{K}(\omega)$, as follows.
For almost every $\omega\in \La^c$ there is a cube $Q^\omega\in
\D^{\widehat{K}}(\omega)$ that contains $Q$.
Note that $a\in L^1$ and has compact support and $f\in L_{loc}^1$. It is easy to see that
$$
\E_\omega\left(\sum_{\alpha}\Big\langle \Big|e^{-\ell(Q_{\alpha}^K(\omega))^2L}f\chi_{Q_{\alpha}^K(\omega)}\Big|,|a|\Big\rangle\right)<\infty.
$$
It follows from $\PP(\La)=0$ that
\begin{align*}
\Big\langle \E_\omega\Big(\sum_{\alpha}e^{-\ell(Q_{\alpha}^K(\omega))^2L}f\chi_{Q_{\alpha}^K(\omega)}\Big),a\Big\rangle&=\E_\omega\left(\sum_{\alpha}\Big\langle e^{-\ell(Q_{\alpha}^K(\omega))^2L}f\chi_{Q_{\alpha}^K(\omega)},a\Big\rangle\right)\\
&=\int_{\La^c}\sum_{\alpha}\Big\langle e^{-\ell(Q_{\alpha}^K(\omega))^2L}f\chi_{Q_{\alpha}^K(\omega)},a\Big\rangle d\omega\\
&=\int_{\La^c}\Big\langle e^{-\ell(Q_{\alpha,\widehat K}^\omega)^2L}f\chi_{Q_{\alpha,\widehat K}^\omega},a\Big\rangle d\omega\\
&=\int_{\La^c}\Big\langle e^{-\ell(Q_{\alpha,\widehat K}^\omega)^2L}f,a\Big\rangle d\omega.
\end{align*}
Let $a=L^2b$ as in the definition  of  $(1,\infty, 2)$. It is easy to check that $\ell(Q)^{-2}Lb$ is $(1,\infty,1)$ atom and thus
$\|\ell(Q)^{-2}Lb\|_{H_L^1}\leq 1$, which shows
\begin{align*}
\Big|\Big\langle e^{-\ell(Q_{\alpha,\widehat K}^\omega)^2L}f,a\Big\rangle\Big|&=\Big|\Big\langle Le^{-\ell(Q_{\alpha,\widehat K}^\omega)^2L}f(x),Lb\Big\rangle\Big|\\
&\leq \frac{\ell(Q)^2}{\ell(Q_{\alpha,\widehat K}^\omega)^2}\Big|\Big\langle \ell(Q_{\alpha,\widehat K}^\omega)^2Le^{-\ell(Q_{\alpha,\widehat K}^\omega)^2L}f(x),\ell(Q)^{-2}Lb\Big\rangle\Big|\\
&\leq \frac{\ell(Q)^2}{\ell(Q_{\alpha,\widehat K}^\omega)^2} \|f\|_{\bmo}\|\ell(Q)^{-2}Lb\|_{H_L^1}\\
&\leq \frac{\ell(Q)^2}{\ell(Q_{\alpha,\widehat K}^\omega)^2} \|f\|_{\bmo},
\end{align*}
since the multiplier operator $t^2Le^{-t^2L}$ is bounded on $H_L^1$ and by duality on $\bmo$ uniformly for $t>0$, see for example \cite[Theorem 1.1]{DY3}. Thus for any $\varepsilon>0$, taking large enough $K$ gives
$$
|{\rm II}|=\Big|\Big\langle \E_\omega\Big(\sum_{\alpha}e^{-\ell(Q_{\alpha}^K(\omega))^2L}f\chi_{Q_{\alpha}^K(\omega)}\Big),a\Big\rangle\Big|\leq \frac{1}{100}\varepsilon.
$$

It follows from \eqref{e4.2} and estimates for ${\rm I}$ and ${\rm II}$ that we only need to show
$$
\Big|\Big\langle \E_\omega\Big(\sum_\alpha h_{Q_{\alpha}^K(\omega)}\chi_{Q_{\alpha}^K(\omega)}\Big),a\Big\rangle\Big|\leq C\varepsilon.
$$
For each $\omega\in \La^c$, we follow the proof of the term ${\rm II}$ to find large enough $K=K(\omega)$  and dyadic cube $Q_{\alpha}^K(\omega)$ such that $Q\subset Q_{\alpha}^K(\omega)$. Since $\PP(\La)=0$, it suffices to show
$$
\left|\int_{\La^c}\langle h_{Q_{\alpha}^K(\omega)},a\rangle d\omega\right|\leq C\varepsilon.
$$
Note that
$$
h_{Q_{\alpha}^K(\omega)}=\frac{\varepsilon}{\lambda}\sum_{Q_k^\omega}\alpha_{Q_k^\omega}(x)\chi_{Q_k^\omega}.
$$
So it suffices to show
$$
\left|\int_{\La^c}\Big\langle \sum_{Q_k^\omega}\alpha_{Q_k^\omega}\chi_{Q_k^\omega},a\Big\rangle d\omega\right|\leq C\lambda.
$$
We divide $Q_k^\omega$ to two classes according to its size. When $\ell(Q_k^\omega)\leq \ell(Q)$, the integral of $\omega$ does not help. It follows from \eqref{size condition} that
\begin{align*}
\left|\left\langle \sum_{\ell(Q_k^\omega)\leq \ell(Q)}\alpha_{Q_k^\omega}\chi_{Q_k^\omega},a \right\rangle \right|
&\leq  \sum_{\ell(Q_k^\omega)\leq \ell(Q)}\left|\left\langle  \alpha_{Q_k^\omega}\chi_{Q_k^\omega},a\right\rangle \right|\\
&\leq \sum_{\ell(Q_k^\omega)\leq \ell(Q)}\sup_{x\in Q_k^\omega} |\alpha_{Q_k^\omega}(x)|\,| \langle\chi_{Q_k^\omega},a\rangle|\\
&\leq C\lambda \sum_{\ell(Q_k^\omega)\leq \ell(Q)} \mu(Q_k^\omega\cap Q)^{1/2} \|a\|_{L^2}\\
&\leq C\lambda \sum_{\ell(Q_k^\omega)\leq \ell(Q)} \left(\frac{\mu(Q_k^\omega\cap Q)}{\mu(Q)}\right)^{1/2}.
\end{align*}
Note that $Q_k^\omega$ are $\gamma=2^{-1}$ graded sequence, that is, sparse collection of dyadic cubes. So
$$
\sum_{\ell(Q_k^\omega)\leq \ell(Q)} \left(\frac{\mu(Q_k^\omega\cap Q)}{\mu(Q)}\right)^{1/2}\leq C
$$
and thus
\begin{align*}
\left|\left\langle \sum_{\ell(Q_k^\omega)\leq \ell(Q)}\alpha_{Q_k^\omega}(x)\chi_{Q_k^\omega},a\right\rangle\right|\leq C\lambda.
\end{align*}

When $\ell(Q_k^\omega)> \ell(Q)$, we need the integral of $\omega$. Note that when $\ell(Q_k^\omega)> \ell(Q)$, for each level of dyadic cubes, there are only finite dyadic cubes intersected with $Q$. So we can use the level index $k$, that is, the size is $2^k$, to label $Q_k^\omega$. Then we change the order of summation of the level index $k$ and the integration of $\omega$ to get
\begin{align*}
&\left|\int_{\La^c} \left\langle \sum_{\ell(Q_k^\omega)>\ell(Q)}\alpha_{Q_k^\omega}(x)\chi_{Q_k^\omega},a \right \rangle d\omega \right|\\
&=\left|\int_{\La^c}\sum_{\ell(Q)< k\leq K(\omega)} \left \langle \sum_{\ell(Q_k^\omega)=2^k}\alpha_{Q_k^\omega}(x)\chi_{Q_k^\omega},a
\right\rangle d\omega\right|\\
&\leq \sum_{k>\ell(Q)}\int_{\La^c}|\langle\alpha_{Q_k^\omega}(x)\chi_{Q_k^\omega},a\rangle| \, d\omega\\
&\leq \sum_{k>\ell(Q)}\int_{{\La_k}}|\langle\alpha_{Q_k^\omega}(x)\chi_{Q_k^\omega},a\rangle| \, d\omega+\sum_{k>\ell(Q)}\int_{\La_k^c}|\langle\alpha_{Q_k^\omega}(x)\chi_{Q_k^\omega},a\rangle| \, d\omega\\
&=:{\rm III+IV}.
\end{align*}
For the term  ${\rm III}$, since $k>\ell(Q)$ and $\omega\in \La_k$, then $Q\subset Q_k^\omega$ and thus
\begin{align*}
 \sum_{k>\ell(Q)}\int_{{\La_k}}|\langle\alpha_{Q_k^\omega}(x)\chi_{Q_k^\omega},a\rangle |d\omega&\leq  \sum_{k>\ell(Q)}\int_{{\La_k}}|\langle\alpha_{Q_k^\omega}(x),a\rangle |d\omega=  \sum_{k>\ell(Q)}\int_{{\La_k}}|\langle \alpha_{Q_k^\omega}(x),L^2b\rangle |d\omega\\
 &=  \sum_{k>\ell(Q)} \frac{1}{2^{2k}}\int_{{\La_k}}|\langle (2^{2k}L)\alpha_{Q_k^\omega}(x),Lb\rangle| d\omega\\
 &\leq  \sum_{k>\ell(Q)} \frac{1}{2^{2k}} \sup_{x\in Q,\omega\in \La_k} |(2^{2k}L)\alpha_{Q_k^\omega}(x)|\cdot \|Lb\|_{L^1}\\
  &\leq  C \lambda \sum_{k>\ell(Q)} \frac{\ell(Q)^2}{2^{2k}}\\
  &\leq  C \lambda,
\end{align*}
where we used the estimate~\eqref{smooth condition} for $L\alpha_{Q_k^\omega}$.

For ${\rm IV}$, now we do not have $Q\subset Q_k^\omega$, but the integration of $\omega$ are small, that is,
\begin{align*}
 \sum_{k>\ell(Q)}\int_{\La_k^c}|\langle\alpha_{Q_k^\omega}(x)\chi_{Q_k^\omega},a\rangle |d\omega
  &\leq  \sum_{k>\ell(Q)} \sup_{x\in Q_k^\omega,\omega\in \La^c_k} |\alpha_{Q_k^\omega}(x)|\cdot \|a\|_{L^1}\cdot \PP(\La_k^c) \\
 &\leq  C \lambda \sum_{k>\ell(Q)} \PP(\La_k^c)\\
 &\leq  C \lambda \sum_{k>\ell(Q)} \left(\frac{\ell(Q)}{2^k}\right)^\eta\\
 &\leq  C \lambda.
\end{align*}
The proof of Theorem~\ref{mainth1} is complete.
\end{proof}

\medskip

In the end of this section, we give a direct application of our Theorem~\ref{th1.1}. Consider  the
Schr\"odinger operator
\begin{eqnarray}
	L=-\Delta+V(x)\ \ \ \ {\rm on}\ {\mathbb R}^n, \ \ \ n\geq 3,
	\label{e1.14}
\end{eqnarray}
where  $V\not\equiv 0$   is a nonnegative potential.  We assume that
$V$ belongs to reverse H\"older class  $  B_q$  for  some $q\geq n/2$,
that is,      $V$ is locally integrable and
\begin{equation}\label{e1.15}
	\left(\frac{1}{ |B|}\int_B V(y)^q~dy\right)^{1/q}\leq\frac{C}{|B|}\int_BV(y)~dy, \ \ \ \ {\rm for\ every\ ball \ }
	B.
\end{equation}
Then the operator   $ L$ is a self-adjoint
operator on $L^2(\RR)$.
From the Feynman-Kac formula, it is well-known that the    kernel $p_t(x,y)$
of the semigroup
$e^{-tL}$  satisfies the estimate
\begin{eqnarray}
	0\leq p_t(x,y)\leq {1\over{ (4\pi t)^{n/2} }} e^{-{{|x-y|^2}\over 4t}}.
	\label{e1.16}
\end{eqnarray}
However,
unless  $V$ satisfies additional conditions, the heat kernel can be
a discontinuous function of the
space variables and the H\"older continuity estimates may fail to
hold (see, for example, \cite{Da}).

By the Fefferman--Stein theorem \cite[Theorem 1.7]{DZ},  it is seen that
for $f\in {\rm BMO}_L({\mathbb R^n})$,
\begin{eqnarray}\label{e1.17}
	f=g_0 +\sum_{j=1}^n    R_j^{\ast} g_j
\end{eqnarray}
where $R_j= \partial_{x_j} L^{-1/2}$ is the Riesz transform associated with $L=-\Delta+V$, and where $g_0, g_1, \cdots, g_n\in L^{\infty}$ with $\|f\|_{\ast}\sim \sum_{j=0}^n \|g_j\|_{\infty}$. Theorem~\ref{th1.1}   can be formulated as

\begin{coro}\label{coro1.2}   Assume that $L=-\Delta +V(x)$ on ${\mathbb R}^n,   n\geq 3$, where
	$V$ is a nonnegative nonzero  function with $
	V\in B_q$  for  some $q\geq n/2$. Then
	there are constants $c_1$ and $c_2$  depending only on the dimension   such that
	\begin{eqnarray*}
	c_1\varepsilon_L(f)  \leq \inf \Big\{  \sum_{j=1}^n \|g_j\|_{\infty}:\
		\eqref{e1.17}\ {\rm holds\ for\ some\ }  g_0\in L^{\infty}  \Big\}  \leq   c_2\varepsilon_L(f)
	\end{eqnarray*}
\end{coro}

We also note that the condition $V\in  B_q$  for  some $q\geq n/2$ may not be necessary for certain special Schr\"odinger operators. One typical example is the Bessel operator $S_\lambda$, $\lambda>0$, introduced by Muckenhoupt--Stein \cite{MS} as follows:
\begin{align}\label{defn of Besl Schr opr}
S_\lambda f(x):=-\frac{d^2}{dx^2}f(x)+\frac{\lambda^2-\lambda}{x^2}f(x),\,x>0.
\end{align}
The operator $S_\lambda$  in \eqref{defn of Besl Schr opr} is a positive self-adjoint operator on $\mathbb R_+:=(0,\infty)$ and it can  be
written in divergence form as
$$S_\lambda=-x^{-\lambda}\ \frac{d}{dx}\ x^{2\lambda}\ \frac{d}{dx}\ x^{-\lambda}=:A_\lambda^\ast A_\lambda,$$
where $A_\lambda:=x^\lambda \ \frac{d}{dx}\ x^{-\lambda}$ and $A_\lambda^\ast:= -x^{-\lambda}\ \frac{d}{dx}\  x^\lambda$ is the adjoint operator of $A_\lambda$. The heat kernel of $S_\lambda$ satisfies the Gaussian upper bound and ${\rm BMO}_{S_\lambda}(\mathbb R_+)$ (the BMO space associated with $S_\lambda$) has the decomposition as in \eqref{e1.17} via the Riesz transform $R_{\lambda}= A_\lambda S^{-1/2}_\lambda$. That is, for
$f\in {\rm BMO}_{S_\lambda}(\mathbb R_+)$, $f=g_0 +    R_\lambda g_1$ with $g_0, g_1\in L^{\infty}$ and
$\|f\|_{{\rm BMO}_{S_\lambda}(\mathbb R_+)}\sim \sum_{j=0}^1 \|g_j\|_{\infty}$. Thus, Corollary \ref{coro1.2} holds for $L=S_\lambda$ on $\mathbb R_+$.

\medskip

\section{Proof of Theorem~\ref{th1.3}}
\setcounter{equation}{0}

To  prove Theorem~\ref{th1.3}, we assume that  $f\in {\rm BMO}_L(X)$ has compact support. For dyadic system $\omega\in \Omega$, we denote by $Q_\alpha^K(\omega)\in \mathcal D^K(\omega)$ the dyadic cube with level $K$ in this dyadic system.
 Noting that $\cup_\alpha Q_{\alpha}^K(\omega)=X$, we have
  $$
 f(x)=\sum_\alpha f(x)\chi_{Q_\alpha^K(\omega)}(x).
 $$
For each $Q_{\alpha}^K(\omega)$, it follows from Proposition~\ref{lem3.2} that
$$
\left(f(x)-e^{-\ell(Q_{\alpha}^K(\omega))^2L}f(x)\right)\chi_{Q_\alpha^K(\omega)}(x)=
g_{Q_{\alpha}^K(\omega)}(x)+h_{Q_{\alpha}^K(\omega)}(x).
$$
It implies that
\begin{eqnarray}\label{e55.1}
f(x)=\sum_{\alpha}\left(e^{-\ell(Q_{\alpha}^K(\omega))^2L}f(x)\chi_{Q_{\alpha}^K(\omega)}(x)+
g_{Q_{\alpha}^K(\omega)}(x)\chi_{Q_{\alpha}^K(\omega)}(x)+h_{Q_{\alpha}^K(\omega)}(x)\chi_{Q_{\alpha}^K(\omega)}(x)\right).
\end{eqnarray}
Let   $\varepsilon=\lambda=C_0\|f\|_{\BL(X)}$. From Proposition~\ref{lem3.2},
$$\|g_{Q_{\alpha}^K(\omega)}\|_\infty<C \|f\|_{\BL(X)}\ \ \ \
{\rm and}\ \ \ \ h_{Q_{\alpha}^K(\omega)}=C_0\sum_{Q^\omega} a_{Q^\omega}\chi_{Q^\omega},
$$
where   $a_{Q^\omega}$
  satisfies estimates~\eqref{size condition3.1},  \eqref{smooth condition3.3}, \eqref{smooth condition3.2} and \eqref{support condition3.4},
  and $\{Q^\omega\}$ is the $2^{-1}$-graded sequence in $Q_{\alpha}^K(\omega)$
   and $Q^\omega$ are as in Proposition~\ref{lem3.2}.


 Since $f\in L_{\rm loc}^1(X)$   and has compact support, $$e^{-\ell(Q_{\alpha}^K(\omega))^2L}f(x)\to 0
 $$
  uniformly for $x$ when $K\to \infty$. Therefore, there exits a positive number $K_0=K_0(f)$, independent of $\omega, x$, such that for $K>K_0$, $\|\sum_{\alpha}e^{-\ell(Q_{\alpha}^K(\omega))^2L}f\chi_{Q_{\alpha}^K(\omega)}\|_{L^\infty}
  \leq \frac{1}{8}\|f\|_{\BL(X)}$.
   We  denote by $Q_{\rm max}$ a  ball which contains $\bigcup_{\alpha,\omega}\{Q_{\alpha}^K(\omega): \ Q_{\alpha}^K(\omega)\cap \supp f\neq \emptyset\}$.

 We take the expectation over $\omega\in \Omega$ in \eqref{e55.1} to obtain
\begin{eqnarray}\label{ekk}
f(x)&=&g(x)+h(x),
\end{eqnarray}
where
\begin{align}\label{def:g}
g(x)= \E_\omega\left(\sum_{\alpha}\left(e^{-\ell(Q_{\alpha}^K(\omega))^2L}f(x)\chi_{Q_{\alpha}^K(\omega)}(x)+
g_{Q_{\alpha}^K(\omega)}(x)\chi_{Q_{\alpha}^K(\omega)}(x)\right)\right)
\end{align}
and
\begin{align}\label{def:h}
h(x)=\sum_{k\leq K} f_k(x)=\sum_{k\leq K} \E_\omega\left(\sum_{\alpha}\sum_{\ell(Q_\alpha^\omega)=2^k} a_{Q_\alpha^\omega}(x)\chi_{Q_\alpha^\omega}(x)\right).
\end{align}
It is clear that  $\|g\|_\infty<C \|f\|_{\BL(X)}$. Define
\begin{eqnarray*}
    &&\La = \{\omega\in\Omega:\mbox{ there is no cube\ } Q^\omega\in\D(\omega)
        \mbox{\ that contains\ the\ support\ of\ } f   \}.
\end{eqnarray*}
By condition~\eqref{Rcondition},
$$
\PP(\La)=0.
$$
If $\omega\in \La^c$, there exist $K=K(\omega)$  and dyadic cube $Q_{\alpha}^K(\omega)$ such that the support of $f$ is included in $ Q_{\alpha}^K(\omega)$. From this point of view, we can assume that there is just one $\alpha$ in the summation of \eqref{def:h} when we take expectation over $\omega\in \Omega$.
That is,
$$
h(x)=\sum_{k\leq K} f_k(x)=\sum_{k\leq K} \E_\omega\left(\sum_{\ell(Q^\omega)=2^k} a_{Q^\omega}(x)\chi_{Q^\omega}(x)\right).
$$
Recall $a_{Q^\omega}$
  satisfies estimates~\eqref{size condition3.1},  \eqref{smooth condition3.3}, \eqref{smooth condition3.2} and \eqref{support condition3.4} and $\{Q^\omega\}$ is the $2^{-1}$-graded sequence as in Proposition~\ref{lem3.2}.

Theorem~\ref{th1.3} will proved by an iteration argument. First,
we have the following result.
\begin{lemma}\label{le6.1}
	Let $\{f_k\}$ be a family of functions in \eqref{ekk}.
For every $\delta>0$, we denote by  $d\sigma_k$   surface measure on
	$X\times \{t=\delta 2^k\}$. Then
	$$
	d\sigma=\sum_{k\leq K} f_k(x)\,d\sigma_k
	$$
   is a Carleson measure with $\|\sigma\|_{\mathcal C}\leq C \delta^{-n}\|f\|_{\BL(X)} $.	
\end{lemma}

\begin{proof} Note that for any $Q$,  it follows from the sparse property of $Q^\omega$ that
\begin{align}\label{Carlesonproof}
 {1\over {\mu(Q)}}\int_{Q\times [0,\ell(Q))}  d\sigma &=\frac{1}{\mu(Q)}\int_{Q}\sum_{\delta 2^k\leq \ell(Q)} f_k(x)\,d\mu(x)\nonumber\\
&\leq \frac{1}{\mu(Q)} \E_\omega \left(\sum_{ \ell(Q^\omega)\leq \delta^{-1}\ell(Q)}\mu(Q^\omega\cap Q)\left\|a_{Q^\omega}\right\|_{L^\infty(Q^\omega)}\right)\nonumber\\
&\leq  C\delta^{-n} \|f\|_{\BL(X)}.
\end{align}
This proves Lemma~\ref{le6.1}.
\end{proof}

Using Lemma~\ref{le6.1},
\begin{align*}
\int_{X\times [0,\infty)} p_{t^2}(x,y)\,d\sigma(y,t)
&=\sum_{k\leq K} \int_X p_{(\delta 2^k)^2}(x,y)f_k(y)\,d\mu(y)\\
&=\sum_{k\leq K} e^{-\delta^2 2^{2k}L}f_k(x).
\end{align*}
This allows us to  rewrite \eqref{ekk}  as
\begin{align}\label{ekkk}
	f(x)&=g(x)+ \int_{X\times [0,\infty)} p_{t^2}(x,y)\,d\sigma(y,t)
	+\left(h(x)-\int_{X\times [0,\infty)} p_{t^2}(x,y)\,d\sigma(y,t)\right)\nonumber\\
&=g(x)+ \int_{X\times [0,\infty)} p_{t^2}(x,y)\,d\sigma(y,t)
	+\sum_{k\leq K} \left(I-e^{-\delta^2 2^{2k}L}\right)f_k(x),
\end{align}
where
\begin{align}\label{ee-Balayage}
\|g\|_{L^\infty}+\|\sigma\|_{\mathcal{C}}\leq C \|f\|_{\BL(X)}.
\end{align}

From the Gaussian upper bound \textbf{(H2)},  we use
  the argument as in \eqref{Carlesonproof} to get
\begin{align}\label{e4.1+}
\left|\sum_{k\leq K} \left(I-e^{-\delta^2 2^{2k}L}\right)f_k(x) \chi_{(2Q_{\rm max})^c}(x)\right|
&= \left|\sum_{k\leq K}e^{-\delta^2 2^{2k}L}f_k(x)\chi_{(2Q_{\rm max})^c}(x)\right|\nonumber\\
&\leq \sum_{k\leq K} \left(\frac{\delta 2^k}{2^K}\right)\int_{Q_{\rm max}} \frac{|f_k(y)|}{\mu(Q_{\rm max})}d\mu(y)\nonumber\\
&\leq  \delta\frac{1}{\mu(Q_{\rm max})}\int_{Q_{\rm max}} \sum_{ 2^k\leq \ell(Q_{\rm max})}|f_k(y)| \,d\mu(y)\nonumber\\
&\leq C\delta \|f\|_{\BL(X)}\nonumber\\
&\leq \frac{1}{4} \|f\|_{\BL(X)}.
\end{align}
Note that  the above estimates \eqref{ee-Balayage} and \eqref{e4.1+} do not depend on $K$. Therefore, if we show that
\begin{align}\label{e4.1}
\left\|\sum_{k\leq K} \left(I-e^{-\delta^2 2^{2k}L}\right)f_k(x)\right\|_{\BL(X)}\leq \frac{1}{2} \|f\|_{\BL(X)}\ \ \  {\rm \ for } \ \delta\ {\rm small},
\end{align}
then  the proof of Theorem~\ref{th1.3}  is proved  by  an iteration
argument.
The proof of  \eqref{e4.1} follows from the duality argument and the following Lemmas~\ref{lem5.1}, \ref{lem5.2} and  \ref{lem5.3}.

\begin{lemma}\label{lem5.1}
Let
\begin{align}\label{f_k}
	f_k(x):= \E_\omega \left(\sum_{\ell(Q^\omega)=2^k} a_{Q^\omega}(x)\chi_{Q^\omega}(x)\right).
\end{align}
 Then there exists a large enough constant $A>0$ such that for any $(1,\infty,1)$-atom $a=Lb$ associated with $L$,  with support $Q$ and for any number $0<\delta<1$, we have
\begin{eqnarray}\label{e6.mm}
\sum_{2^k >A\ell(Q) }\left|\left\langle (I-e^{-\delta^2 2^{2k}L})f_k,a\right\rangle\right|\leq \frac{1}{12}\|f\|_{\bmo}.
\end{eqnarray}
\end{lemma}
\begin{proof}
One has
$$
{\rm LHS\ of \ } \eqref{e6.mm} \leq \sum_{2^k>A\ell(Q)} \left|\left\langle f_k,a \right\rangle\right| +\sum_{2^k>A\ell(Q)}  \left|\left\langle e^{-\delta^2 2^{2k}L}f_k, a \right\rangle\right|=:
{\rm I}+{\rm II}.
$$

Recall that
\begin{eqnarray*}
    &&\La_k := \{\omega\in\Omega: \mbox{ there exists a cube }Q^\omega\in\D^{k}(\omega)
        \mbox{ such that }Q\subset Q^\omega\}.
\end{eqnarray*}
Then we obtain that
\begin{align*}
{\rm I}
&\leq  \sum_{2^k>A\ell(Q)} \int_{\La_k } \Big|\Big\langle \sum_{\ell(Q^\omega)=2^k} a_{Q^\omega}\chi_{Q^\omega},a \Big\rangle\Big| \, d\omega
+\sum_{2^k>A\ell(Q)} \int_{\La^c_k } \Big|\Big\langle \sum_{\ell(Q^\omega)=2^k} a_{Q^\omega}\chi_{Q^\omega},a \Big\rangle\Big| \, d\omega\\
&=:{\rm I}_{1}+{\rm I}_{2}.
\end{align*}
For $\omega\in \La_k$ and $2^k >A\ell(Q)$, we denote by $Q^\omega_0$ the unique cube  such that  $Q^\omega_0\in\D^{k}(\omega)
        \mbox{ and }Q\subset Q^\omega_0$. Then we apply the properties of $(1,\infty,1)$-atom associated with $L$ to get
\begin{align*}
{\rm I}_{1}\leq  \sum_{2^k>A\ell(Q)} \int_{\La_k } |\langle a_{Q^\omega_0},a \rangle|\, d\omega
&\leq \sum_{2^k>A\ell(Q)}  \frac{1}{2^{2k}}\| (2^{2k} L)a_{Q^\omega_0}\|_{L^\infty(Q^\omega_0)}\|b \|_{L^1}\\
&\leq C\|f\|_{\BL(X)}\sum_{2^k>A\ell(Q)} \Big(\frac{\ell(Q)}{2^k}\Big)^2\\
&\leq C A^{-\varepsilon} \|f\|_{\BL(X)}.
\end{align*}
For the term ${\rm I}_{2}$, we use \eqref{Rcondition} to get
\begin{align*}
{\rm I}_{2}&\leq \sum_{2^k>A\ell(Q)} \PP(\La^c_k) \sup_{\ell(Q^\omega)=2^k}\| a_{Q^\omega}\|_{L^\infty(Q^\omega)}\|a\|_{L^1}\\
&\leq C\|f\|_{\BL(X)}\sum_{2^k>A\ell(Q)} \left(\frac{\ell(Q)}{2^k}\right)^{\varepsilon}\\
&\leq C A^{-\varepsilon} \|f\|_{\BL(X)}
\end{align*}
and thus ${\rm I} \leq C A^{-\varepsilon} \|f\|_{\BL(X)}.$

Next we  estimate the term ${\rm II}$. One can write
\begin{align*}
{\rm II}
&\leq  \sum_{2^k>A\ell(Q)} \int_{\La_k } \Big|\Big\langle  e^{-\delta^2 2^{2k}L}\Big(\sum_{\ell(Q^\omega)=2^k} a_{Q^\omega}\chi_{Q^\omega}\Big),a \Big\rangle\Big| d\omega\\
&\quad\quad+\sum_{2^k>A\ell(Q)} \int_{\La^c_k }  \Big|\Big\langle  e^{-\delta^2 2^{2k}L}\Big(\sum_{\ell(Q^\omega)=2^k} a_{Q^\omega}\chi_{Q^\omega}\Big),a \Big\rangle\Big| d\omega\\
&=:{\rm II}_{1}+{\rm II}_{2}.
\end{align*}

For the term ${\rm II}_{2}$, we use the similar argument as the proof of ${\rm I}_{2}$ to get
\begin{align*}
{\rm II}_{2}&\leq \sum_{2^k>A\ell(Q)} \PP(\La^c_k) \sup_\omega
\Big\|e^{-\delta^2 2^{2k}L}\Big(\sum_{\ell(Q^\omega)=2^k} a_{Q^\omega}\chi_{Q^\omega}\Big)\Big\|_{L^\infty}\|a\|_{L^1}\\
&\leq \sum_{2^k>A\ell(Q)} \PP(\La^c_k) \sup_{\omega,Q^\omega}\|a_{Q^\omega}\|_{L^\infty(Q^\omega)}\|a\|_{L^1}\\
&\leq C\sum_{2^k>A\ell(Q)} \left(\frac{\ell(Q)}{2^k}\right)^{\varepsilon}\|f\|_{\BL(X)}\\
&\leq C A^{-\varepsilon} \|f\|_{\BL(X)}.
\end{align*}
To estimate  the term ${\rm II}_{1}$, we let $Q_0^\omega$ be the dyadic cube which contains $Q$.  We write
\begin{align*}
{\rm II}_{1}
&\leq \sum_{2^k>A\ell(Q)} \int_{\La_k } \Big|\Big\langle  e^{-\delta^2 2^{2k}L}\Big(\sum_{\ell(Q^\omega)=2^k,Q^\omega\neq Q_0^\omega} a_{Q^\omega}\chi_{Q^\omega}\Big),a \Big\rangle\Big| \, d\omega\\
&\quad\quad +\sum_{2^k>A\ell(Q)} \int_{\La_k } \Big|\Big\langle  e^{-\delta^2 2^{2k}L}\Big( a_{Q_0^\omega}\chi_{(Q_0^\omega)^c}\Big),a \Big\rangle\Big|\, d\omega\\
&\quad\quad +\sum_{2^k>A\ell(Q)} \int_{\La_k } \left|\left\langle  e^{-\delta^2 2^{2k}L}\big( a_{Q_0^\omega}\big),a \right\rangle\right| \,d\omega\\
&=:{\rm II}_{11}+{\rm II}_{12}+{\rm II}_{13}.
\end{align*}
Let us  first estimate ${\rm II}_{13}$. Using  \eqref{support condition3.4}  and properties of $(1,\infty,1)$-atom associated with the operator $L$, we have
\begin{align*}
{\rm II}_{13}
&\leq \sum_{2^k>A\ell(Q)}  \frac{1}{2^{2k}}\left\|e^{-\delta^2 2^{2k}L}\left( (2^{2k}L)a_{Q_0^\omega}\right)\right\|_{L^\infty(Q)}\|b \|_{L^1}\\
&\leq C\|f\|_{\BL(X)}\sum_{2^k>A\ell(Q)} \left(\frac{\ell(Q)}{2^k}\right)^2\\
&\leq C A^{-2} \|f\|_{\BL(X)}.
\end{align*}
For the term  ${\rm II}_{12}$, we consider two cases: $\delta 2^k\geq \ell(Q)$ and $\delta 2^k<\ell(Q)$.
In the case that $\delta 2^k\geq \ell(Q)$, set
\begin{align*}
    &\La_{k,0} := \left\{\omega\in\Omega: d(Q,(Q_0^\omega)^c)<\delta 2^k\right\};\\
    &\La_{k,i} := \left\{\omega\in\Omega: d(Q,(Q_0^\omega)^c)\in[2^{i-1} \delta 2^k, 2^{i}\delta 2^k)\right\}, \ \  i=1,2,\ldots, M;\\
    &\La_{k,M+1} := \left\{\omega\in\Omega: d(Q,(Q_0^\omega)^c)\geq 2^{M-1} \delta 2^k\right\},
\end{align*}
where $M$ is the biggest integer smaller than $\log_2(2^k/\delta).$ This, together with  $a=Lb$,  gives
\begin{align*}
{\rm II}_{12}&\leq \sum_{2^k>A\ell(Q)}\sum_{i=0}^{M+1} \int_{\La_{k,i}} \Big|\Big\langle  e^{-\delta^2 2^{2k}L}\Big( a_{Q_0^\omega}\chi_{(Q_0^\omega)^c}\Big),a \Big\rangle\Big|\, d\omega\\
&\leq\sum_{2^k>A\ell(Q)} \sum_{i=0}^{M+1} \PP(\La_{k,i})\sup_{\omega\in \La_{k,i}} \frac{1}{\delta^2 2^{2k}}\Big|\Big\langle  (\delta^2 2^{2k}L)e^{-\delta^2 2^{2k}L}\big( a_{Q_0^\omega}\chi_{(Q_0^\omega)^c}\big),b \Big\rangle\Big|.
\end{align*}
By \eqref{Rcondition} and \eqref{support condition3.4},
\begin{align*}
&\sum_{2^k>A\ell(Q)} \PP(\La_{k,0})\sup_{\omega\in \La_{k,0}} \frac{1}{\delta^2 2^{2k}}\Big|\Big\langle  (\delta^2 2^{2k}L)e^{-\delta^2 2^{2k}L}\big( a_{Q_0^\omega}\chi_{(Q_0^\omega)^c}\big),b \Big\rangle\Big|\\
&\leq \sum_{2^k>A\ell(Q)} \left(\frac{\delta 2^k}{2^k}\right)^\varepsilon\frac{\|b\|_{L^1}}{\delta^2 2^{2k}}\sup_{\omega\in \La_{k,i}}
\left\|\int_{X} |h_{(\delta^2 2^{2k})}(x,y)||a_{Q_0^\omega}(y)|\,d\mu(y)\right\|_{L^\infty(Q)}\\
&\leq \sum_{2^k>A\ell(Q)} \delta^\varepsilon \left(\frac{\ell(Q)}{\delta 2^k}\right)^2 \|f\|_{\BL(X)}\\
&\leq C A^{-\varepsilon} \|f\|_{\BL(X)},
\end{align*}
where in the last inequality we used the condition $\delta 2^k\geq \ell(Q)$.
We then apply \eqref{Rcondition}, \eqref{e3.10} and \eqref{support condition3.4} to obtain
\begin{align*}
&\sum_{2^k>A\ell(Q)}\sum_{i=1}^{M+1} \PP(\La_{k,i})\sup_{\omega\in \La_{k,i}} \frac{1}{\delta^2 2^{2k}}\Big|\Big\langle  (\delta^2 2^{2k}L)e^{-\delta^2 2^{2k}L}\big( a_{Q_0^\omega}\chi_{(Q_0^\omega)^c}\big),b \Big\rangle\Big|\\
&\leq \sum_{2^k>A\ell(Q)} \sum_{i=1}^{M+1} \left(\frac{2^i\delta 2^k}{2^k}\right)^\varepsilon\frac{\|b\|_{L^1}}{\delta^2 2^{2k}}\sup_{\omega\in \La_{k,i}}
{\left\|\int_{(Q_0^\omega)^c} K_{(\delta^2 2^{2k}L)e^{-\delta^2 2^{2k}L}}(x,y)a_{Q_0^\omega}(y)\,d\mu(y)\right\|_{L^\infty(Q)}
}\\
&\leq C\sum_{2^k>A\ell(Q)} \sum_{i=1}^{M+1} \left({2^i\delta }\right)^\varepsilon\left(\frac{\ell(Q)}{\delta 2^k}\right)^22^{-Ni}\sup_{\omega\in \La_{k,i}}
{\left\|\int_{X} |h_{(2\delta^2 2^{2k})}(x,y)||a_{Q_0^\omega}(y)|\,d\mu(y)\right\|_{L^\infty(Q)}
}\\
&\leq C \sum_{2^k>A\ell(Q)} \sum_i \left(\frac{2^i\delta 2^k}{2^k}\right)^\varepsilon \left(\frac{\ell(Q)}{\delta 2^k}\right)^2 2^{-Ni}\|f\|_{\BL(X)}\\
&\leq C A^{-\varepsilon} \|f\|_{\BL(X)},
\end{align*}
where $N$ is any integer and in the last inequality above we used the condition  $\delta 2^k\geq \ell(Q)$.

Now in the case that  $\delta 2^k\leq \ell(Q)$,  set
\begin{align*}
    &\tilde \La_{k,0} := \{\omega\in\Omega: d(Q,(Q_0^\omega)^c)< \ell(Q)\};\\
    &\tilde \La_{k,i} := \{\omega\in\Omega: d(Q,(Q_0^\omega)^c)\in [2^{i-1} \ell(Q), 2^{i}\ell(Q))\}, \ i=1,\ldots, \tilde{M};\\
    &\tilde \La_{k,\tilde{M}+1} := \{\omega\in\Omega: d(Q,(Q_0^\omega)^c)\geq 2^{\tilde{M}}\ell(Q)\},
\end{align*}
where $\tilde{M}$ is the biggest integer smaller than $\log_2({2^k/\ell(Q)})$.
Then
\begin{align*}
{\rm II}_{12}
&\leq \sum_{2^k>A\ell(Q)} \sum_{i=0}^{\tilde{M}+1} \PP(\tilde\La_{k,i})\sup_{\omega\in \tilde\La_{k,i}}\left\|\int_{(Q_0^\omega)^c} K_{e^{-\delta^2 2^{2k}L}}(x,y)a_{Q_0^\omega}(y)\, d\mu(y)\right\|_{L^\infty(Q)}\\
&\leq C \sum_{2^k>A\ell(Q)} \left[ \left(\frac{\ell(Q)}{2^k}\right)^\varepsilon+ \sum_{i=1}^{\tilde{M}+1} \left(\frac{2^i\ell(Q)}{2^k}\right)^\varepsilon \left(\frac{2^i\ell(Q)}{\delta 2^k}\right)^{-N}\right] \|f\|_{\BL(X)}\\
&\leq C A^{-\varepsilon} \|f\|_{\BL(X)},
\end{align*}
where in the last inequality above we used the condition $\delta 2^k\leq \ell(Q)$.  From the estimates in the above  two cases,    we see that  ${\rm II}_{12}\leq
  C A^{-\varepsilon} \|f\|_{\BL(X)}$.

Finally, let us  estimate  the term  ${\rm II}_{11}$. Since there are  finite $\{Q^\omega\}$  satisfing  ${\rm dist}(Q^\omega,Q^\omega_0)\leq 2^k$,  the similar argument as that in the estimate of ${\rm II}_{12}$ gives  the estimate for these $Q^\omega$. Hence,
it  remains   to compute the summation of $Q^\omega$ which satisfies ${\rm dist}(Q^\omega,Q^\omega_0)>2^k$. We have
\begin{align*}
{\rm II}_{11}&\leq\sum_{2^k>A\ell(Q)}
\sup_{\omega}\frac{1}{\delta^2 2^{2k}} \left|\left\langle\sum_{\ell(Q^\omega)=2^k,{\rm dist}(Q^\omega,Q^\omega_0)>2^k}
(\delta^2 2^{2k}L)e^{-\delta^2 2^{2k}L}\left( a_{Q^\omega}\chi_{Q^\omega}\right), b \right\rangle\right|\\
&\leq \sum_{2^k>A\ell(Q)} \left(\frac{1}{\delta 2^{k}}\right)^2 \sup_{\omega} \left\|\int_{{\rm dist}(x,Q)>2^k} |K_{\delta^2 2^{2k}Le^{-\delta^2 2^{2k}L}}(x,y)| \, d\mu(y)\right\|_{L^\infty(Q)} \|a_{Q^\omega}\|_{L^\infty(Q^\omega)}\|b\|_{L^1}\\
&\leq \sum_{2^k>A\ell(Q)}\left(\frac{\ell(Q)}{\delta 2^{k}}\right)^2 \delta^{N} \|f\|_{\BL(X)},   \quad \mbox{for any large number N},\\
&\leq C A^{-2} \|f\|_{\BL(X)}.
\end{align*}
This, together with estimates of ${\rm II}_{12}$ and  ${\rm II}_{13}$,
yields that ${\rm II}_{1}\leq C A^{-\varepsilon} \|f\|_{\BL(X)}$, and  so ${\rm II} \leq C A^{-\varepsilon} \|f\|_{\BL(X)}$.
Therefore, we combine estimates of ${\rm I}$ and ${\rm II} $  to obtain
$$\sum_{2^k >A\ell(Q) }\left|\left\langle (I-e^{-\delta^2 2^{2k}L})f_k,a\right\rangle\right|\leq C A^{-\varepsilon} \|f\|_{\BL(X)}\leq { 1\over  12} \|f\|_{\BL(X)},
$$
whenever $A$ is chosen to be a constant large enough. This completes the  proof of Lemma~\ref{lem5.1}.
\end{proof}

\begin{lemma}\label{lem5.2}
Let $f_k$ be as \eqref{f_k} and large number $A$ be as Lemma \ref{lem5.1}. Then there exists  a small positive constant $\delta$ such that  for arbitrary $(1,\infty,1)$-atom associated with $L$, $a=Lb$, with support $Q$,
$$
\sum_{\ell(Q) <2^k\leq A\ell(Q)}\Big|\Big\langle (I-e^{-\delta^2 2^{2k}L})f_k,a\Big\rangle\Big|\leq \frac{1}{12}\|f\|_{\bmo}.
$$
\end{lemma}
\begin{proof}
Recall that $f_k(x)=\E_\omega \left(\sum_{\ell(Q^\omega)=2^k} a_{Q^\omega}(x)\chi_{Q^\omega}(x)\right)$.
Then
\begin{eqnarray*}
 &&\hspace{-1.5cm}\sum_{\ell(Q) <2^k\leq A\ell(Q)} \Big|\Big\langle \left(I-e^{-\delta^2 2^{2k}L}\right)f_k,a \Big\rangle  \Big|\\
&\leq& \sum_{\ell(Q) <2^k\leq A\ell(Q)}\Big|\Big\langle \E_\omega\Big(\Big(I-e^{-\delta^2 2^{2k}L}\Big)\Big( \sum_{\substack{\ell(Q^\omega)=2^k\\d(Q^\omega,Q)\leq 2^k}} a_{Q^\omega}\chi_{Q^\omega}\Big)\Big),a \Big\rangle\Big|\\
&&+\sum_{\ell(Q) <2^k\leq A\ell(Q)}\E_\omega\left(\Big|\Big\langle \Big(I-e^{-\delta^2 2^{2k}L}\Big)\Big( \sum_{\substack{\ell(Q^\omega)=2^k\\d(Q^\omega,Q)> 2^k}} a_{Q^\omega}\chi_{Q^\omega}\Big),a \Big\rangle\Big|\right)\\
& =:&{\rm II}_1+{\rm II}_2.
\end{eqnarray*}
For ${\rm II}_2$, note that $Q^\omega\cap Q=\emptyset$ when $\ell(Q^\omega)=2^k$ and $d(Q^\omega,Q)> 2^k$. Thus,
\begin{align*}
{\rm II}_2
&\leq \|a_{Q^\omega}\|_{L^\infty(Q^\omega)}\sum_{\ell(Q) <2^k\leq A\ell(Q)}\sup_\omega \left\|\int_{d(y,Q)>2^k}|K_{e^{-\delta^2 2^{2k}L}}(\cdot,y)| \, d\mu(y)\right\|_{L^\infty(Q)}\|a\|_{L^1}\\
&\leq C\sum_{\ell(Q) <2^k\leq A\ell(Q)} \left(\frac{2^k}{\delta 2^k}\right)^{-N}\|f\|_{\BL(X)}, \ \quad \mbox{for large N}\\
&\leq C\delta\log A \|f\|_{\BL(X)},
\end{align*}
where in the last inequality we used $0<\delta<1$. Then we may take $\delta$ small enough such that $C\delta\log A<1/6$.

Consider ${\rm II}_1$.  Since the number of $Q^\omega$ satisfying  $\ell(Q^\omega)=2^k$ and $d(Q^\omega,Q)< 2^k$ is  a constant which depends only on the doubling constants in \eqref{doubling2},  we only need to consider one $Q^\omega$.

For any fixed $x$,  if $x\notin Q^\omega$, then $x$ must be in some $ Q_0^\omega\in\D^{k}(\omega)$. Denote
\begin{align*}
&\Theta_{k,0} := \Big\{\omega\in\Omega: d(x, (Q_0^\omega)^c)< \delta 2^k \Big\};\\
&\Theta_{k,i} := \Big\{\omega\in\Omega: x \mbox{ is in a dyadic ring $[2^{i-1} \delta 2^k, 2^{i}\delta 2^k)$ of the edge of }Q_0^\omega\Big\}, \ i=1,\ldots,M;\\
&\Theta_{k,M+1} := \Big\{\omega\in\Omega: d(x, Q_0^\omega)\geq   2^{M}\delta 2^k\Big\},
\end{align*}
where $M$ is the largest integer that is smaller than $\log 1/(2\delta)$.
Then for  $x\notin Q^\omega$, it holds
\begin{align}\label{ee-x-not-in}
\Big|\E_\omega\Big((I-e^{-\delta^2 2^{2k}L})( a_{Q^\omega}\chi_{Q^\omega})(x)\Big)\Big|
&\leq \sum_{i=0}^{M+1} \int_{\Theta_{k,i}}\Big |e^{-\delta^2 2^{2k}L}( a_{Q^\omega}\chi_{Q^\omega})(x)\Big| \, d\omega\nonumber\\
&\leq \sum_{i=0}^{M+1} \PP(\Theta_{k,i}) \sup_\omega \int_{Q^\omega} |h_{\delta 2^k}(x,y)||a_{Q^\omega}(y)|d\mu(y)\nonumber\\
&\leq C\left(\delta^\varepsilon+\sum_{i=1}^{M+1} \left(\frac{2^i\delta 2^k}{2^k}\right)^\varepsilon   \left(\frac{2^i\delta 2^k}{\delta 2^k}\right)^{-N}\right)\|f\|_{\BL(X)}, \  \ \mbox{for large N}\nonumber\\
&\leq C \delta^\varepsilon \|f\|_{\BL(X)}.
\end{align}
If $x\in Q^\omega$, we write
\begin{align}\label{eII_1}
&\E_\omega\left((I-e^{-\delta^2 2^{2k}L})( a_{Q^\omega}\chi_{Q^\omega})(x)\right)&
\nonumber\\
&= \E_\omega\left((I-e^{-\delta^2 2^{2k}L})( a_{Q^\omega})(x)\right)-\E_\omega\left((I-e^{-\delta^2 2^{2k}L})( a_{Q^\omega}\chi_{(Q^\omega)^c})(x)\right).
\end{align}
The first term on the right side of \eqref{eII_1} follows from \eqref{smooth condition3.2} that
\begin{align}\label{ee-x-in-1}
\Big|\E_\omega\left((I-e^{-\delta^2 2^{2k}L})( a_{Q^\omega})(x)\right)\Big|
&\leq C\delta^2  \|f\|_{\BL(X)}.
\end{align}
For the second term on the right side of \eqref{eII_1}, we use the similar argument as the case  $x\notin Q^\omega$. For $i=0,1, \cdots M+1$, denote
\begin{eqnarray*}
    &&\Theta_{k,0} := \Big\{\omega\in\Omega:  d(x, (Q^\omega)^c)<\delta 2^k\Big\},\\
    &&\Theta_{k,i} := \Big\{\omega\in\Omega: x \mbox{ is in a dyadic ring $[2^{i-1} \delta 2^k, 2^{i}\delta 2^k)$ of the edge of }Q^\omega\Big\},\\
    &&\Theta_{k,M+1} := \Big\{\omega\in\Omega:  d(x, (Q^\omega)^c)\geq 2^M\delta 2^k\Big\},
\end{eqnarray*}
where $M$ is the biggest integer smaller than $\log_2 1/(2\delta)$.
 So for any $x\in Q^\omega$,  it follows from \eqref{support condition3.4} of properties of $a_{Q^\omega}$ that
\begin{align}\label{ee-x-in-2}
\big|\E_\omega\left((I-e^{-\delta^2 2^{2k}L})( a_{Q^\omega}\chi_{(Q^\omega)^c})(x)\right)\big|
&\leq \sum_{i=0}^{M+1} \int_{\Theta_{k,i}} \big|e^{-\delta^2 2^{2k}L}( a_{Q^\omega}\chi_{(Q^\omega)^c})(x)\big|\, d\omega\nonumber\\
&\leq \sum_{i=0}^{M+1} \PP(\Theta_{k,i}) \sup_\omega {\int_{(Q^\omega)^c} |h_{\delta^2 2^{2k}}(x,y)||a_{Q^\omega}(y)|\,dy}\nonumber\\
&\leq C\left(\delta^\varepsilon+\sum_{i=1}^{M+1} \left(\frac{2^i\delta 2^k}{2^k}\right)^\varepsilon  \left(\frac{2^i\delta 2^k}{\delta 2^k}\right)^{-N}\right)\|f\|_{\BL(X)}, \ \ \mbox{for large N}\nonumber\\
&\leq C \delta^\varepsilon \|f\|_{\BL(X)}.
\end{align}
Combining \eqref{ee-x-not-in}--\eqref{ee-x-in-2},  we have showed that
$$
{\rm II}_1
\leq C \delta^\varepsilon \|f\|_{\BL(X)} \log A<\frac{1}{6}\|f\|_{\BL(X)},
$$
by taking $\delta$ small enough. Therefore, the proof of Lemma \ref{lem5.2} is complete.
\end{proof}

\begin{lemma}\label{lem5.3}
Let $f_k$ be given in \eqref{f_k} and constant $A$ be as in Lemma \ref{lem5.1}. Then there exists  a small positive constant $\delta$ such that  for arbitrary $(1,\infty,1)$-atom $a=Lb$ associated with $L$,  with support $Q$,
$$
\sum_{2^k\leq \ell(Q)}\Big|\Big\langle (I-e^{-\delta^2 2^{2k}L})f_k,a\Big\rangle\Big|\leq \frac{1}{12}\|f\|_{\bmo}.
$$
\end{lemma}
\begin{proof}

Recall that $f_k(x)=\E_\omega \left(\sum_{\ell(Q^\omega)=2^k} a_{Q^\omega}(x)\chi_{Q^\omega}(x)\right)$. In the proof of this lemma, we  will  not  use the average of $\omega$.
Let $Q^{(0)}=CQ$ be the ball with an absolute constant $C>1$  such that for all $\omega\in \Omega$, if $Q^\omega$ satisfies $\ell (Q^\omega)\sim\ell( Q)$ and $Q^\omega\cap Q\neq \emptyset$, then $Q^\omega \subset Q^{(0)}$.  Pave $X$ with cubes $Q^{(j)}$ ``congruent" to $Q^{(0)}$. They may intersect each other, but they are bounded overlap. The balls $\{Q^j\}$ also have the following property: for any $\omega\in \Omega$, if $\ell(Q^\omega)\leq \ell(Q)$, there exists a $Q^{(j)}$ such that $Q^\omega\subset Q^{(j)}$.
Otherwise, we can double  all initial $Q^{(j)}$. Then
\begin{align}\label{estiamtesforIII0}
\sum_{2^k< \ell(Q)}\Big|\Big\langle \left(I-e^{-\delta^2 2^{2k}L}\right)f_k,a \Big\rangle\Big|
&= \sum_{2^k<\ell(Q)}\Big|\Big\langle \E_\omega\Big(\Big(I-e^{-\delta^2 2^{2k}L}\Big)\Big( \sum_{\ell(Q^\omega)=2^k} a_{Q^\omega}\chi_{Q^\omega}\Big)\Big),a \Big\rangle\Big|\nonumber\\
&\leq \E_\omega \left(\sum_{j} \sum_{Q^\omega\subset Q^{(j)}}\Big|\Big\langle\Big(I-e^{-\delta^2 \ell(Q^\omega)^2L}\Big)\Big(a_{Q^\omega}\chi_{Q^\omega}\Big),a \Big\rangle\Big|\right).
\end{align}

If $Q^\omega\subset Q^{(j)}$, $j\neq 0$, we have $Q^\omega\cap Q=\emptyset$ and thus
\begin{align}\label{estiamtesforIII1}
 \Big|\Big\langle(I-e^{-\delta^2 \ell(Q^\omega)^2L})(a_{Q^\omega}\chi_{Q^\omega}),a \Big\rangle\Big|
&\leq   \sup_{x\in Q}\int_{Q^\omega}|K_{e^{-\delta^2 \ell(Q^\omega)^2L}}(x,y)| \,d\mu(y)\|a_{Q^\omega}\|_{L^\infty(Q^\omega)}\|a\|_{L^1}\nonumber\\
&\leq C\big\|f\big\|_{\BL(X)} \left(\frac{d(Q,Q^\omega)}{\delta \ell(Q^\omega)}\right)^{-n-1} \sup_{y\in Q^\omega} \frac{1}{\mu((y,\delta \ell(Q^\omega)))}\mu(Q^\omega)\nonumber\\
&\leq \delta\big\|f\big\|_{\BL(X)} \frac{\ell(Q)^{n+1}}{d(Q,Q^{(j)})^{n+1}}\frac{\mu(Q^\omega)}{\mu(Q^{(j)})} .
\end{align}

If $Q^\omega\subset Q^{(0)}$, then
\begin{align}\label{estiamtesforIII2}
 &\Big|\Big\langle\big(I-e^{-\delta^2 \ell(Q^\omega)^2L}\big)\big(a_{Q^\omega}\chi_{Q^\omega}\big),a \Big\rangle\Big|
\leq \Big\|\big(I-e^{-\delta^2 \ell(Q^\omega)^2L}\big)\big(a_{Q^\omega}\chi_{Q^\omega}\big)\Big\|_{L^1(Q)}\|a\|_{L^\infty}\nonumber\\
&\leq \frac{1}{\mu(Q)}\left(\Big\|\big(I-e^{-\delta^2 \ell(Q^\omega)^2L}\big)\big(a_{Q^\omega}\chi_{Q^\omega}\big)\Big\|_{L^1((Q^\omega)^c)}
+\Big\|\big(I-e^{-\delta^2 \ell(Q^\omega)^2L}\big)\big(a_{Q^\omega}\chi_{Q^\omega}\big)\Big\|_{L^1(Q^\omega)}\right)\nonumber\\
&=:\frac{1}{\mu(Q)}\Big({\rm III}_1+{\rm III}_2\Big).
\end{align}
For the term ${\rm III}_1$, we decompose $(Q^\omega)^c\backslash (2Q^\omega)^c$ into the  following annulus $\{Q^{\omega}_i\}_{i=0}^M$, where $M\sim{\log \delta^{-1}}$, $Q^{\omega}_0:=\{x: d(x,Q^\omega)\leq  \delta\ell(Q^\omega)\}$, and  $Q^{\omega}_i:=\{x: d(x,Q^\omega)\sim  2^{i}\delta\ell(Q^\omega)\}$, $i=1,\cdots,M$. Then we have
\begin{align*}
{\rm III}_1
&\leq C\left\|\int_{Q^\omega} |K_{e^{-\delta^2 \ell(Q^\omega)^2L}}(\cdot,y)a_{Q^\omega}(y)|\,d\mu(y)\right\|_{L^1((Q^\omega)^c)}\\
&\leq  C\|f\|_{\BL(X)}\left( \left\|\int_{Q^\omega} |K_{e^{-\delta^2 \ell(Q^\omega)^2L}}(\cdot,y)|\,d\mu(y)\right\|_{L^1((2Q^\omega)^c)}+
\sum_{i= 0}^{\log \delta^{-1}}\left\|\int_{Q^\omega} |K_{e^{-\delta^2 \ell(Q^\omega)^2L}}(\cdot,y)|\, d\mu(y)\right\|_{L^1(Q^\omega_i)}\right)\\
&\leq C \|f\|_{\BL(X)}\left( \mu(Q^\omega)\delta+
\sum_{i= 0}^{\log \delta^{-1}}2^{-iN}\mu(Q^\omega_i)\right)\\
&\leq  C\|f\|_{\BL(X)} \mu(Q^\omega)\delta^{\eta},
\end{align*}
where in the last inequality above we used the fact $\mu(Q^\omega_i)\leq C2^{i\eta} \delta^\eta \mu(Q^\omega)$ by \eqref{annulars2}.

For the term ${\rm III}_2$,
\begin{align*}
{\rm III}_2&\leq \Big\|\big(I-e^{-\delta^2 \ell(Q^\omega)^2L}\big)a_{Q^\omega}\Big\|_{L^1(Q^\omega)}
+\Big\|\big(I-e^{-\delta^2 \ell(Q^\omega)^2L}\big)\big(a_{Q^\omega}\chi_{(Q^\omega)^c}\big)\Big\|_{L^1(Q^\omega)}\\
&=:{\rm III}_{21}+{\rm III}_{22}.
\end{align*}
It  follows from \eqref{smooth condition3.2} that
\begin{align*}
{\rm III}_{21}
&\leq \|f\|_{\BL(X)}\delta^2 \mu(Q^\omega).
\end{align*}
For ${\rm III}_{22}$, the proof is very similar to that of ${\rm III}_1$.
Decompose $(Q^\omega)^c\backslash (2Q^\omega)^c$ to the  annulus $Q^{\omega}_i$ such that $d(Q^{\omega}_i, Q^\omega)\sim 2^i\delta\ell(Q^\omega)$
\begin{align*}
{\rm III}_{22}
&\leq {\Big\|\int_{(2Q^\omega)^c} |K_{e^{-\delta^2 \ell(Q^\omega)^2L}}(\cdot,y)a_{Q^\omega}(y)|\,d\mu(y)\Big\|_{L^1(Q^\omega)} }
+\sum_{i= 0}^{\log \delta^{-1}}\Big\|\int_{Q^\omega_i}
 |K_{e^{-\delta^2 \ell(Q^\omega)^2L}}(\cdot,y)a_{Q^\omega}(y)|\,d\mu(y)\Big\|_{L^1(Q^\omega)}\\
&\leq \delta \|f\|_{\BL(X)} \mu(Q^\omega)+ \|a_{Q^\omega}(y)\|_{L^\infty(2Q^\omega)}\sum_{i= 0}^{\log \delta^{-1}}\Big\|\int_{Q^\omega_i} |K_{e^{-\delta^2 \ell(Q^\omega)^2L}}(\cdot,y)|\,d\mu(y)\Big\|_{L^1(Q^\omega)}\\
&\leq C \delta^\eta \|f\|_{\BL(X)} \mu(Q^\omega),
\end{align*}
where in the second inequality we used \eqref{support condition3.4}.
Combing estimates of ${\rm III}_1$ and ${\rm III}_2$, we substitute \eqref{estiamtesforIII1} and \eqref{estiamtesforIII2} to \eqref{estiamtesforIII0} to see
\begin{align*}
&\sum_{2^k< \ell(Q)}\Big|\Big\langle (f_k(x)-e^{-\delta^2 2^{2k}L}f_k(x)),a \Big\rangle\Big|\\
&\leq \E_\omega \left(\sum_{j\neq 0} \sum_{Q^\omega\subset Q^{(j)}}\delta\|f\|_{\BL(X)} \frac{\ell(Q)^{n+1}}{d(Q,Q^{(j)})^{n+1}}\frac{\mu(Q^\omega)}{\mu(Q^{(j)})}\right)+\E_\omega\left( \sum_{Q^\omega\subset Q^{(0)}}\delta^\eta \|f\|_{\BL(X)} \frac{\mu(Q^\omega)}{\mu(Q)}\right)\\
&\leq C\delta^\eta \|f\|_{\BL(X)} .
\end{align*}
This completes the proof of Lemma~\ref{lem5.3}.
\end{proof}

\medskip

\section{Further discussions  }
\setcounter{equation}{0}

In this section, we discuss some directions related to Theorems~\ref{th1.1} and \ref{th1.3} on the space ${\rm BMO}_L(X)$ associated with operators,  which might be helpful to further studies.

 1) When the  underlying space $(X,d,\mu)$ is Ahlfors regular, i.e., $\mu$ satisfies
\begin{eqnarray}\label{e8.11}
	cr^n \leq \mu(B(x,r))\leq C r^n
\end{eqnarray}
	for all $r>0$ uniformly for $x\in X$,
	 our  Theorems \ref{th1.1}  and \ref{th1.3}   still hold  when the heat semigroup $e^{-tL}$ is repalced by  the Poisson semigroup $e^{-t\sqrt L}$ of the operator $L$. The results can be proved with minor variations in the argument in the proof of Theorems~\ref{th1.1}  and \ref{th1.3},
	and we omit the detail here.

2) Recall that a weight $w(x)\in (A_2)$  if
$$
\sup_B \left({1\over |B|}\int_B wdx \right)\left({1\over |B|}\int_B {1\over w} dx \right)<\infty.
$$
It is known that $w\in (A_2)$ if and only if the Riesz transforms are bounded on $L^p(wdx)$ (see for instance, \cite{St2}).
As pointed out in \cite{GJ1},  a consequence  of the theorem of Garnett--Jones \cite{GJ1}
is a higher dimensional Helson--Szeg\"o theorem. They proved
the following result.

\begin{proposition}\label{prop8.1}
Suppose there is a positive constant $B_1(n)$  such that
\begin{eqnarray}\label{e8.1}
	f=g_0 +\sum_{j=1}^n R_j g_j
\end{eqnarray}
 with $\sum_{j=1}^n \|g_j\|_{\infty}<B_1(n)$, then
$e^{f}$ satisfies  $(A_2)$. Conversely, if $e^{f}\in (A_2)$, then \eqref{e8.1} holds for $f$ with
$\sum_{j=1}^n \|g_j\|_{\infty}<B_2(n)$, where the constant $B_2(n)$ depends only on $n$.
\end{proposition}

Suppose that $L$ is an operator
on $L^2(X)$ satisfying \textbf{(H1)} and \textbf{(H2)}.
Let $w\in {\mathcal M}$ be a positive  weight function on ${\mathbb R^n}$.
We say a weight $w(x)\in (A_{2, L})$  if
$$
\sup_t  \left\| e^{-tL}  ( w)
e^{-tL}  ( w^{-1})   \right\|_{L^{\infty}(X)}
  <\infty.
$$

It was proved in \cite[Theorem 0.8]{PV} that the classical
$(A_2)$ coincides with the one associated with the standard Laplacian  $\Delta$ on $\mathbb R^n$, that is,  $(A_2)= (A_{2, \Delta})$, see also \cite{DKP}.

Our results in this article suggest that it is possible
to establish an analogous of Proposition~\ref{prop8.1}  when the classical
Riesz transforms  are replaced by the Riesz transforms associated with operator $L$. We will study the weights and $\BMO$ spaces associated with operators in future work.


 \bigskip

\noindent
{\bf Acknowledgements}: The authors thank Guoqian Wang for helpful discussions.
 P. Chen, L. Song and L. Yan   were supported  by National Key R$\&$D Program of China 2022YFA1005700.
  P. Chen was supported by NNSF of China 12171489, Guangdong Natural Science Foundation 2022A1515011157. X.T. Duong was supported by  the Australian Research Council (ARC) through the research
 grant DP190100970.  J. Li  was supported   by ARC DP 220100285.  L. Song was supported by NNSF of China 12071490.

\end{document}